\documentclass[11pt]{article}
\usepackage[top=2cm,bottom=2.5cm,right=2.5cm,left=2.5cm]{geometry}
\usepackage[english]{babel}
\usepackage[T1]{fontenc}
\usepackage{times}
\usepackage{mathtools, bm}
\usepackage{amssymb, bm}
\usepackage{graphicx}
\usepackage{mathrsfs}
\usepackage{stmaryrd}
\usepackage{amsthm}
\usepackage{listings}
\usepackage{pict2e}
\usepackage{pgf, tikz}
\usepackage{dsfont}
\usepackage{algorithm2e}
\usepackage{algorithmic}
\usepackage{multicol}
\usepackage{hyperref}

\usepackage[export]{adjustbox}

\newcommand{\pierre}[1]{{\color{black}#1}}

\newcommand{\plb}[1]{{\color{black}#1}}

\allowdisplaybreaks

\thicklines

{
      \theoremstyle{plain}
      
  }

\newtheorem{theorem}{Theorem}

\newtheorem{lemma}{Lemma}
\newtheorem{remark}{Remark}

\numberwithin{equation}{section} 
\numberwithin{lemma}{section} 
\numberwithin{remark}{section} 
\numberwithin{example}{section}
\numberwithin{corollary}{section}
\numberwithin{proposition}{section}

\date{}
\author{Arnaud Guillin\footnote{Laboratoire de Math\'{e}matiques Blaise Pascal  - Universit\'{e} Clermont-Auvergne and Institut Universitaire de France. Email : arnaud.guillin[AT]uca.fr},$\ $ Pierre Le Bris\footnote{ Laboratoire Jacques-Louis Lions - Sorbonne Université. Email : pierre.lebris[AT]sorbonne-universite.fr}$\ $ and Pierre Monmarch\'{e}\footnote{ Laboratoire Jacques-Louis Lions - Sorbonne Université. Email : pierre.monmarche[AT]sorbonne-universite.fr}}
\title{Some remarks on the effect of the Random Batch Method on phase transition}

\begin{document}

\maketitle

\centerline{\it In memory of Francis Comets}

\begin{abstract}
In this article, we focus on two toy models : the \textit{Curie-Weiss} model and the system of $N$ particles in linear interactions in a \textit{double well confining potential}. Both models, which have been extensively studied, describe a large system of particles with a mean-field limit that admits a phase transition. We are concerned with the numerical simulation of these particle systems. To deal with the quadratic complexity of the numerical scheme, corresponding to the computation of the $O(N^2)$ interactions per time step, the \textit{Random Batch Method} (RBM) has been suggested. It consists in randomly (and uniformly) dividing the particles into batches of size $p>1$, and computing the interactions only within each batch, thus reducing the numerical complexity to $O(Np)$ per time step. The convergence of this numerical method has been proved in other works. 

This work is motivated by the observation that the RBM, via the random constructions of batches, artificially adds noise to the particle system. The goal of this article is to study the effect of this added noise on the phase transition of the nonlinear limit, and more precisely we study the \textit{effective dynamics} of the two models to show how a phase transition may still be observed with the RBM but at a lower critical temperature.  
\end{abstract}

%\tableofcontents

%
%
%Section
%
%

\section{Introduction}

%
%
%Subsection
%
%

\subsection{Motivation}

Consider a system of $N$ particles $(X^i)_{i\in\{1,...,N\}}$ in interaction
\begin{equation}\label{eq:IPS}
\tag{IPS}
    dX^i_t=-\nabla U(X^i_t)dt-\frac{1}{N-1}\sum_{j\neq i}\nabla W(X^i_t-X^j_t)dt+\sqrt{2\sigma}dB^i_t,
\end{equation}
where for all $i\in\{1,...,N\}$ and $t\geq0$ we have $X^i_t\in\mathbb{R}^d$, $U$ and $W$ are two twice continuously differentiable functions, respectively called \textit{confining potential} and \textit{interaction potential}, $\sigma>0$ is a \textit{diffusion coefficient} or \textit{temperature}, and $(B^i)_i$ are independent $d$-dimensional Brownian motions. The name \eqref{eq:IPS} refers to \textit{Interacting Particle System}.

It is well known (see \cite{CD22-1,CD22-2} and references therein) that, under suitable assumptions on $U$ and $W$, the particle system \eqref{eq:IPS} converges as $N\rightarrow\infty$ towards its nonlinear mean-field limit, a stochastic differential equation (SDE) of \textit{McKean-Vlasov} type
\begin{equation}\label{eq:NL}
\tag{NL}
\left\{
\begin{array}{ll}
    d\bar{X}_t=-\nabla U(\bar{X}_t)dt-\nabla W\ast \bar{\rho}_t(\bar{X}_t)dt+\sqrt{2\sigma}dB_t,\\
    \bar{\rho}_t=\text{Law}(\bar{X}_t).
\end{array}   
\right.
\end{equation}
Here, the name \eqref{eq:NL} refers to \textit{Nonlinear Limit}, and this equation arises in the modelling of granular media \cite{CGM08}.

The quantitative link between of \eqref{eq:IPS} and \eqref{eq:NL} can be exploited in various ways. On one hand, as it was historically motivated, the study of (way too) large systems of particles cannot be feasible, and boiling it down to the study of the nonlinear limit yields exploitable results. On the other hand, one can see \eqref{eq:IPS} as an approximation of \eqref{eq:NL}, and in particular an approximation that can be numerically simulated. Consider the \textit{Euler-Maruyama} scheme associated to \eqref{eq:IPS} with a timestep $\delta>0$
\begin{equation}\label{eq:D-IPS}
\tag{D-IPS}
\left\{
\begin{array}{ll}
    X^{i,\delta}_{t+1}=X^{i,\delta}_{t}-\delta \nabla U(X^{i,\delta}_{t})-\frac{\delta}{N-1}\sum_{j\neq i}\nabla W(X^{i,\delta}_{t}-X^{j,\delta}_{t})+\sqrt{2\sigma\delta}G^i_t,\\
    G^i_t\ \text{ i.i.d }\ \sim\mathcal{N}(0,1),\ \ \ t\in\mathbb{N}.
\end{array}  
\right.
\end{equation}
Its name \eqref{eq:D-IPS} comes from \textit{Discrete - Interacting Particle System}. The convergence of \eqref{eq:D-IPS} towards \eqref{eq:NL} has been extensively studied : with bounded Lipschitz coefficients \cite{BT97}, with Hölder continuous coefficients \cite{BH19}, non-Lipschitz coefficients \cite{DQ21}. The quantitative convergence of the implicit Euler-Maruyama scheme can also be found in \cite{Mal03}.

Notice that this numerical scheme requires $O(N^2)$ operations per time step, corresponding to the total number of interactions of pairs $(i,j)_{i,j\in\{1,...,N\}}$. To cope with this possibly limiting complexity, several works have suggested using the \textit{Random Batch Method} (RBM) (see for instance \cite{JLL20}), motivated by the Stochastic Gradient Langevin Dynamics \cite{WT11}.

Consider, for a time step $t\in\mathbb{N}$, a partition  $\mathcal{P}_t=\left(\mathcal{P}^1_t,...,\mathcal{P}^{N/p}_t\right)$ of $\{1,...,N\}$ into $N/p$ subsets of size $p>1$, assuming for the sake of simplicity that $N$ is a multiple of $p$, and define
\begin{equation}\label{eq:def_rb_clust}
    \mathcal{C}^i_t=\left\{j\in\{1,..,N\}\text{ s.t. }\exists l\in\{1,...,N/p\}, i,j\in\mathcal{P}^{l}_t\right\}.
\end{equation}
In other words, $\mathcal{C}^i_t$ is the set of indexes that are in the same subset as $i$ at time step $t$, with the convention $i\in\mathcal{C}^i_t$. We now consider the following numerical scheme
\begin{equation}\label{eq:D-RB-IPS}
\tag{D-RB-IPS}
\left\{
\begin{array}{ll}
    Y^{i,\delta,p}_{t+1}=Y^{i,\delta,p}_{t}-\delta \nabla U(Y^{i,\delta,p}_{t})-\frac{\delta}{p-1}\sum_{j\in\mathcal C^i_t\setminus\{i\}}\nabla W(Y^{i,\delta,p}_{t}-Y^{j,\delta,p}_{t})+\sqrt{2\sigma\delta}G^i_t,\\
    G^i_t\ \text{ i.i.d }\ \sim\mathcal{N}(0,1),\ \ \ i\in \{1,...,N\},\ \ \ t\in\mathbb{N},
\end{array}  
\right.
\end{equation}
where for each time step $t$ the partition $\mathcal{P}_t$ is random and each partition has the same probability of occurring. The name \eqref{eq:D-RB-IPS} refers to \textit{Discrete - Random Batch - Interacting Particle System}. The convergence of \eqref{eq:D-RB-IPS} towards \eqref{eq:NL} can be found in \cite{JLL21, JLYZ22, YZ22}.

The idea of using random batches has been shown to be efficient for computing the evolution of large interacting system of quantum particles \cite{GJP21}, of particles with Coulomb interactions in molecular dynamics \cite{JLXZ21}, but also for Markov Chain Monte Carlo \cite{LXZ20}, or for solving PDEs \cite{CJT22, LLT22}. See also references therein.

The starting point of this work is the following observation : the RBM, via the random construction of a partition of $\{1,...,N\}$, artificially adds noise (or temperature) to a system. We thus ask the following question :
\begin{center}
    \textit{Does the critical temperature of (the mean-field limit of) a system of interacting particles admitting a phase transition decrease when considering a version with random batches ? If so, can we quantify it ?}
\end{center} 
To partially answer this question, we focus on two specific types of particle systems for which the mean-field limit admits a phase transition : the first one is the \textit{Curie-Weiss} model and the second one is the system \eqref{eq:IPS} in dimension 1 with attractive and quadratic interaction potential $W$ and the \textit{double well confining potential} $U$.

The nonlinear mean-field limits of both models admit, as we will discuss, a phase transition occurring at a certain critical parameter. We consider a version with random batches of size $p$ of each system, consider the limit as $N\rightarrow\infty$ (with fixed $p$) towards a nonlinear model, and then study the phase transition of said limit.

%
%
%Subsection
%
%

\subsection{The Curie-Weiss model}

\paragraph{The classical system.} The Curie-Weiss model is, and it is the reason we start by studying it, arguably one of the most simple system admitting a phase transition. Consider $N$ spins, given by a configuration $\sigma=(\sigma_1,...,\sigma_N)$, and $\Omega_N=\{-1,1\}^N$ the set of possible configurations for the system. On this system we consider the following Hamiltonian
\begin{equation}\label{eq:CW_def_H}
\forall \sigma\in\Omega_N,\ \ \ H_N(\sigma)=-\frac{1}{2N}\sum_{i,j}\sigma_i\sigma_j.
\end{equation}
Intuitively, each spin will tend to align with the others. It is a mean field model as $H_N$ only depends in reality on the mean magnetization $m_N(\sigma):=\frac{1}{N}\sum_{i=1}^N\sigma_i$, by
\begin{equation*}
    H_N(\sigma)=-\frac{N}{2}m_N(\sigma)^2.
\end{equation*}
The evolution for $(\sigma(n))_{n\geq0}$ in $\Omega_N$ is the following : at each discrete time step, a spin is chosen uniformly among the $N$ possible spins. Let us denote $i$ this spin, and $\sigma'=(\sigma'_1,...,\sigma'_N)$ the configuration such that for all $j\neq i$, $\sigma'_j=\sigma(n)_j$, and $\sigma'_i=-\sigma(n)_i$. We accept $\sigma'$ as the next step of $\sigma(n)$ with probability $\exp\left(-\beta(H_N(\sigma')-H_N(\sigma))_+\right)$ (i.e if the Hamiltonian decreases then with probability 1, otherwise with a positive probability depending on a parameter $\beta$), otherwise the system remains at $\sigma(n)$. \plb{Here we use the notation $x_+=\text{max}(x,0)$.} This parameter $\beta$ is known as the \textit{inverse temperature}. This yields the following transition probabilities for the Markov chain $(\sigma(n))_{n\geq0}$:
\begin{equation*}
p(\sigma,\sigma')=\left\{
\begin{array}{ll}
\frac{1}{N}\exp\left(-\beta(H_N(\sigma')-H_N(\sigma))_+\right)&\text{ if }||\sigma-\sigma'||_1=2\\
0&\text{ if }||\sigma-\sigma'||_1>2\\
1-\sum_{\eta\neq\sigma}p(\sigma,\eta)&\text{ if }\sigma'=\sigma
\end{array}
\right.
\end{equation*}
This dynamics $(\sigma(n))_{n\geq0}$, which is an irreducible and aperiodic Markov chain on a finite state space $\Omega_N$, is reversible with respect to the Gibbs measure
\begin{equation}
\mu_{\beta,N}(\sigma)=\frac{1}{Z_{\beta,N}}\exp(-\beta H_N(\sigma)),
\end{equation}
where $Z_{\beta,N}$ is a normalizing constant. Instead of studying the dynamics of $\sigma$, we look at the mean magnetization $m_N(n)=m_N(\sigma(n))$, which is still a Markov chain. This quantity, at each time step, can only increase or decrease by $\frac{2}{N}$, and the transition probabilities are given by
\begin{equation}\label{eq:CW_def_r}
r(m,m')=\left\{
\begin{array}{ll}
\frac{1-m}{2}\exp\left(-\frac{\beta N}{2}(m^2-m'^2)_+\right)&\text{ if }m'=m+\frac{2}{N}\\
\frac{1+m}{2}\exp\left(-\frac{\beta N}{2}(m^2-m'^2)_+\right)&\text{ if }m'=m-\frac{2}{N}\\
1-r\left(m,m+\frac{2}{N}\right)-r\left(m,m-\frac{2}{N}\right)&\text{ if }m'=m\\
0&\text{ otherwise.}
\end{array}
\right.
\end{equation}
Likewise, this dynamics is reversible with respect to the Gibbs measure
\begin{align*}
\nu_{\beta,N}(m)=\frac{1}{Z_{\beta,N}}\binom{N}{\frac{1+m}{2}N}\exp\left(\frac{\beta Nm^2}{2}\right).
\end{align*}
Many works (see for instance \cite{CK17, EN78, LLP10}, the classical reference that is Chapter~4 of \cite{Ell85} or more recently Chapter~2 of \cite{FV17}) have studied Large Deviation Principles for this system, and have shown that there exists a critical inverse temperature $\beta_c=1$. For the sake of completeness, and because the method will be similar in the case with random batches, we give a proof in Section~\ref{sec:CW_no_minibatch} of the phase transition happening in the following sense : the process $M^{(N)}_t=m_N(\lfloor Nt\rfloor)$ weakly converges to the solution of an ordinary differential equation (ODE). For $\beta>1$, the limit ODE admits three equilibrium states, and for $\beta\leq1$ only one. In both cases, 0 is an equilibrium state, and is stable in the case $\beta\leq1$ and unstable in the case $\beta>1$. \plb{Phase transition for the maximum likelihood estimator of the parameters has also been studied in \cite{CG91}.}

\paragraph{The Curie-Weiss model with random batches.} We then consider the same system, but using the Random Batch Method. At each time step, the chosen spin no longer evolves according to the entire system, but according to a subset of $p$ spins containing the chosen spin.

We thus consider a new evolution for $(\sigma^{p}(n))_{n\geq0}$ in $\Omega_N$, where $\sigma^p$ denotes the new sequence of spin configurations. At each discrete time step, a spin is chosen uniformly among the $N$ possible spins. Let us denote it $i$, and $\sigma'=(\sigma'_1,...,\sigma'_N)$ the configuration such that for all $j\neq i$, $\sigma'_j=\sigma^{p}(n)_j$, and $\sigma'_i=-\sigma^{p}(n)_i$. We then sample a subset of $\{1,...,N\}$ of size $p$  containing $i$, denoted $\mathcal{C}^{i,p}$, uniformly over such subsets, and accept $\sigma'$ as the next step of $\sigma^{p}(n)$ with probability $\exp\left(-\beta(H_{N,p}(\sigma',\mathcal{C}^{i,p})-H_{N,p}(\sigma^{p}(n),\mathcal{C}^{i,p}))_+\right)$, where
\begin{equation}\label{eq:RB_CW_def_H}
    H_{N,p}(\sigma,\mathcal{C}^{i,p})=-\frac{1}{2p}\sum_{j,k\in\mathcal{C}^{i,p}}\sigma_j\sigma_k.
\end{equation}
Likewise, we may study this system in terms of its magnetization, denoted $(m_{N,p}(n))_{n}$, for which we can explicitly write the transition probabilities (see Lemma~\ref{lem:trans_prob}).

This system resembles to some extent the \textit{dilute Curie-Weiss model} \cite{BMP21}, in which the spins interact according to an Erd\H{o}s-R\'enyi random graph with edge probability $\tilde{p}=\frac{p}{N}\in]0,1[$, the main difference being that the "graph", in our case, is modified at each time step and there are exactly $p-1$ spins interacting with a given one.

Studying the Curie-Weiss model with random batches, which is done in Section~\ref{sec:CW_minibatch}, yields the following results.

%Theorem

\begin{theorem}\label{thm:RB-CW}
Let $p\in\mathbb{N}\setminus\{0,1\}$ and $\beta>0$. 
\begin{itemize}
    \item Define
    \begin{align*}
        S^{p,\beta}_1(m)&=\sum_{k=0}^{p-1}\binom{p-1}{k}\left(\frac{1-m}{2}\right)^k\left(\frac{1+m}{2}\right)^{p-1-k}e^{-2\beta\left(\frac{2k+1-p}{p}\right)_+ }\\
        S^{p,\beta}_2(m)&=\sum_{k=0}^{p-1}\binom{p-1}{k}\left(\frac{1-m}{2}\right)^k\left(\frac{1+m}{2}\right)^{p-1-k}e^{-2\beta\left(\frac{p-1-2k}{p}\right)_+ },\\
        f_p(\beta,m)=&\left(S^{p,\beta}_1(m)-S^{p,\beta}_2(m)\right)-m\left(S^{p,\beta}_1(m)+S^{p,\beta}_2(m)\right).
    \end{align*}
    The process $M^{(N,p)}_t=m_{N,p}(\lfloor Nt\rfloor)$, i.e the magnetization rescaled in time, weakly converges as $N\rightarrow\infty$ to the solution of the ODE
    \begin{equation}\label{eq:ODE_RB_CW}
        \frac{d}{dt}m(t)=f_p(\beta,m(t)).
    \end{equation}
    For all $\beta>0$, $0$ is an equilibrium state for the solution of \eqref{eq:ODE_RB_CW}. 
    \item For $p\in\{2,3\}$, $0$ is the unique equilibrium state, and it is stable.
    \item For $p\geq4$, there exists $\beta_{c,p}$ such that for all $\beta>\beta_{c,p}$, the equilibrium state 0 is unstable, and for all $\beta\leq\beta_{c,p}$ it is stable. Furthermore, we have the estimate
    \begin{equation}
        \beta_{c,p}=1+\sqrt{\frac{2}{p\pi}}+o\left(\frac{1}{\sqrt{p}}\right).
    \end{equation}
\end{itemize}
\end{theorem}
This theorem thus gives a first answer to the main question of the article : the RBM does increase the critical inverse temperature of the system (i.e decreases the critical temperature).

%
%
%Subsection
%
%

\subsection{Numerical scheme and double-well potential}

We then go back to the initial motivation concerning numerical schemes for interacting particle systems.

\paragraph{The effective dynamics.} Just like we may consider the nonlinear limit of \eqref{eq:IPS}, we may also consider the limit as $N\rightarrow\infty$ of \eqref{eq:D-RB-IPS}. Define
\begin{equation}\label{eq:D-RB-NL}
\tag{D-RB-NL}
\left\{
\begin{array}{ll}
    \bar{Y}^{\delta, p}_{t+1}=\bar{Y}^{\delta, p}_{t}-\delta \nabla U(\bar{Y}^{\delta, p}_{t})-\frac{\delta}{p-1}\sum_{j=1}^{p-1}\nabla W(\bar{Y}^{\delta, p}_{t}-Y^j)+\sqrt{2\sigma\delta}G_t,\\
    G_t\ \text{ i.i.d }\ \sim\mathcal{N}(0,1),\ \ \  (Y^j)_j\ \text{ i.i.d }\ \sim\bar{\rho}^{\delta,p}_t:=\text{Law}(\bar{Y}^{\delta, p}_{t}).
\end{array}  
\right.
\end{equation}
The name \eqref{eq:D-RB-NL} stands for \textit{Discrete - Random Batch - Nonlinear Limit}. The convergence of \eqref{eq:D-RB-IPS} towards \eqref{eq:D-RB-NL} can be found in \cite{JL22}. The proof relies on a coupling method, noticing that, as $N\rightarrow\infty$, the probability of constructing batches of fixed size $p$ in \eqref{eq:D-RB-IPS} with independent and identically distributed particles goes to 1, thus giving a convergence in total variation distance.

We then, in the spirit of \cite{SS21}, construct a continuous process, parameterized by the timestep and the batch size, which is closer to the numerical scheme \eqref{eq:D-RB-IPS} than the target \eqref{eq:NL}. In the dynamics of \eqref{eq:D-RB-NL}, writing
\[\xi_t = \frac{1}{p-1}\sum_{j=1}^{p-1}\nabla W\left(\bar{Y}^{\delta,p}_t-Y^j\right)\]
 notice that
\begin{align*}
    \mathbb{E}\left(\xi_t \Big| \bar{Y}^{\delta,p}_t \right)=  \nabla W\ast \bar{\rho}^{\delta,p}_t(\bar{Y}^{\delta,p}_t),
\end{align*}
and
\begin{align*}
    \text{Var}\left(\xi_t \Big| \bar{Y}^{\delta,p}_t\right)=&\frac{1}{p-1}\text{Var}_{\bar{\rho}^{\delta,p}_t}\left(\nabla W(\bar{Y}^{\delta,p}_t-\cdot)\Big|\bar{Y}^{\delta,p}_t\right)\\
    =&\frac{1}{p-1}\left((\nabla W)^2\ast\bar{\rho}^{\delta,p}_t(\bar{Y}^{\delta,p}_t)-(\nabla W\ast\bar{\rho}^{\delta,p}_t(\bar{Y}^{\delta,p}_t))^2 \right),
\end{align*}
where the square of a vector has to be understood component-wise. Hence,
\[\bar{Y}^{\delta,p}_t = \bar{Y}^{\delta,p}_0 -\delta \sum_{s=0}^{t-1}  \nabla U(\bar{Y}^{\delta, p}_{s}) -\delta \sum_{s=0}^{t-1}\nabla W\ast \bar{\rho}^{\delta,p}_s(\bar{Y}^{\delta,p}_s)  - \delta M_t + \sqrt{2\sigma\delta} \sum_{s=0}^{t-1} G_s \,,\]
where
\[t\mapsto M_t := \sum_{s=0}^{t-1} \left( \xi_s -  \nabla W\ast \bar{\rho}^{\delta,p}_s(\bar{Y}^{\delta,p}_s)\right) \]
is a martingale. By martingale CLT, we thus expect the numerical scheme \eqref{eq:D-RB-IPS} to be close, for small values of $\delta$,  the following non-linear SDE, that we call the \textit{effective dynamics}:
\begin{equation}\label{eq:Eff}
\tag{Eff}
\left\{
    \begin{array}{l}
     d\bar{X}^{e,\delta,p}_t=-\nabla U(\bar{X}^{e,\delta,p}_t)dt-\nabla W\ast\bar{\rho}^{e,\delta,p}_t (\bar{X}^{e,\delta,p}_t)dt+\left(2\sigma +\frac{\delta}{p-1}\Sigma(\bar{X}^{e,\delta,p}_t,\bar{\rho}^{e,\delta,p}_t)\right)^{1/2}dB_t,  \\
     \bar{\rho}^{e,\delta,p}_t=\text{Law}(\bar{X}^{e,\delta,p}_t),
    \end{array}
\right.
\end{equation}
where we denote $\Sigma(x,\rho)=(\nabla W)^2\ast\rho(x)-(\nabla W\ast\rho(x))^2$. Notice that, although it is a continuous-time process, it depends on the stepsize $\delta$ of the numerical schemes.

Such dynamics are also known as modified equations in various works considering the backward error analysis of SDEs \cite{Sha06, Zyg11}, improving upon a technique that had already provided a better understanding of the numerical methods for ODEs. They have been used in the numerical error analysis of the Stochastic Gradient Langevin Dynamics \cite{VZT16, SS21}. Of course, these references do not consider non-linear SDEs as we do, and obtaining a formal result in our case is out of the scope of the present work. Let us informally and briefly explain the motivation of the effective dynamics (we refer to  \cite{VZT16, SS21} and references within for further details). In the usual stochastic gradient case (which would correspond to \eqref{eq:D-RB-NL} where we assume that the law of $Y^j$ is fixed), denote by $\pi_{\delta,p}$, $\pi$, $\pi_\delta$ and $\pi_{eff}$, respectively, the invariant measures of \eqref{eq:D-RB-NL}, of the continuous-time limit process \eqref{eq:NL},  of its Euler scheme (without Random batches, i.e. \eqref{eq:D-RB-NL} with $p=\infty$) and of \eqref{eq:Eff}. For a fixed observable $f$, from the weak error analysis  on the invariant measure  (see e.g. \cite[Proposition 1]{SS21})), we get that there exists $c_1,c_2 \in\mathbb R$ such that  $\pi_{\delta,p}(f) \simeq    \pi(f) + c_1 \delta + c_2 \delta/ p$, while $\pi_{\delta}(f) \simeq  \pi(f) + c_1 \delta$ and $ \pi_{eff}(f)  \simeq  \pi(f) + c_2 \delta/ p$, where these approximations are all up to a term of order $\delta^2 (1+p^{-3/2})$. In other words, at first order in $\delta$, $c_1$ and $c_2$ respectively accounts for the time discretization and stochastic gradient errors. By studying \eqref{eq:Eff}, at first order, we disregard the error which is purely due to the time  discretization and focus on the contribution of the stochastic gradient approximation. Notice that, for the Euler scheme of the overdamped Langevin diffusion \eqref{eq:D-RB-NL}, except if the variance $\Sigma$ is very large (which corresponds to the case in \cite{VZT16, SS21} which are not concerned with a mean-field scaling)  these two parts of the error are of the same order in $\delta$. However, in practice, second-order schemes for underdamped Langevin or Hamiltonian Monte Carlo  are widely used (as in \cite{SS21,monmarche2022hmc}) and in that case the stochastic gradient contribution is the leading term of the bias (see Remark~\ref{rem:HMC} below). In any cases, the numerical scheme is closer to the effective dynamics than it is to the continuous-time process (as they only differ, at first order, through the pure discretization error), which motivates in the following the analysis of the effective dynamics \eqref{eq:Eff}.

Again, we emphasize that providing a quantitative link between the various processes \eqref{eq:IPS}, \eqref{eq:D-IPS}, \eqref{eq:D-RB-IPS}, \eqref{eq:D-RB-NL}, \eqref{eq:Eff}, and \eqref{eq:NL} would require an entire separate analysis, even though some results are already known. As it would dilute the main message of this work concerning the phase transition of the effective dynamics, we do not address this question here.

\paragraph{The double well confining potential.} We now choose in \eqref{eq:NL} the dimension to be $d=1$ and the potentials
\begin{align}
    U(x)=\frac{x^4}{4}-\frac{x^2}{2},\ \ \ W(x)=L_W\frac{x^2}{2}\ \ \text{ with }L_W>0.\label{eq:def_U_W}
\end{align}
Recall the following result adapted from \cite{tugaut}.

%Theorem

\begin{theorem}[Theorem 2.1 of \cite{tugaut}]\label{thm:Tugaut}
For $U$ and $W$ given by \eqref{eq:def_U_W}, there exists $\sigma_c>0$ such that 
\begin{itemize}
    \item For all $\sigma\geq\sigma_c$, there exists a unique stationary distribution $\mu_{\sigma,0}$ for \eqref{eq:NL}. Furthermore, $\mu_{\sigma,0}$ is symmetric.
    \item For all $\sigma<\sigma_c$, there exist three stationary distributions for \eqref{eq:NL}. One is symmetric, also denoted $\mu_{\sigma,0}$, and the other two, denoted $\mu_{\sigma,+}$ and $\mu_{\sigma,-}$, satisfy $\pm\int xd\mu_{\sigma,\pm}(dx)>0$.
\end{itemize}
By convention, in the case $\sigma\geq\sigma_c$, we may denote $\mu_\sigma=\mu_{\sigma,\pm}=\mu_{\sigma,0}$.
\end{theorem}

Our goal is now to study the stationary distribution(s) for the effective dynamics \eqref{eq:Eff} in the specific case of the double-well potential \eqref{eq:def_U_W}. We wish to understand if, similarly as Theorem~\ref{thm:Tugaut}, there exists a phase transition, and if so compare the critical parameters. We thus prove in Section~\ref{sec:RB_for_IPS} the following theorem.

%Theorem

\begin{theorem}\label{thm:phase_trans_eff}
Let $\sigma_0\in]0,\sigma_c[$ where $\sigma_c$ is defined in Theorem~\ref{thm:Tugaut}. For $U$ and $W$ given by \eqref{eq:def_U_W}, there exists $c_0>0$ such that for all $(\delta,p)$ satisfying $\frac{\delta}{p-1}\leq c_0$, denoting
\begin{equation}\label{eq:def_sigma_eff}
    \sigma_c^{eff}=\sigma_c\left(1-\frac{\delta L_W}{2(p-1)}\right),
\end{equation}
we have the following phase transition for the dynamics \eqref{eq:Eff}
\begin{itemize}
    \item For all $\sigma\geq\sigma_c^{eff}$, there exists a unique stationary distribution $\mu_{\sigma,0}^{\delta,p}$ for \eqref{eq:Eff}. Furthermore, $\mu_{\sigma,0}^{\delta,p}$ is symmetric.
    \item For all $\sigma\in[\sigma_0,\sigma_c^{eff}[$, there exists exactly three stationary distributions for \eqref{eq:Eff}. One is symmetric, also denoted $\mu_{\sigma,0}^{\delta,p}$, and the other two, denoted $\mu_{\sigma,+}^{\delta,p}$ and $\mu_{\sigma,-}^{\delta,p}$, satisfy $\pm\int xd\mu_{\sigma,\pm}^{\delta,p}(x)>0$.
\end{itemize}
\end{theorem}

%Remark

\begin{remark}
Let us quickly discuss the form of \eqref{eq:def_sigma_eff}. In the specific case of \eqref{eq:def_U_W}, as discussed in Section~\ref{sec:RB_for_IPS}, one has $\Sigma(\bar{X}^{e,\delta,p}_t,\bar{\rho}^{e,\delta,p}_t)=L_W^2\text{Var}(\bar{\rho}^{e,\delta,p}_t)$. To insist on the dependence on $\sigma$ rather than $(\delta,p)$, let us denote, only in this remark, $\Sigma_\sigma:=\Sigma(\bar{X}^{e,\delta,p}_t,\bar{\rho}^{e,\delta,p}_t)$. 

We will show, but this can be intuitively understood at this stage, that any stationary distribution for \eqref{eq:NL} is a stationary distribution for \eqref{eq:Eff} although for a smaller value of $\sigma$. We thus have to study the stationary distribution at the critical value $\sigma_c$.

As proved in Lemma~\ref{lem:res_NL}, the variance of the stationary distribution for \eqref{eq:NL} at the critical value is $\text{Var}(\mu_{\sigma_c,0})=\frac{\sigma_c}{L_W}$. 

By considering the diffusion term in \eqref{eq:Eff} and ensuring $2\sigma_c=2\sigma_c^{eff}+\frac{\delta}{p-1}\Sigma_{\sigma_c}$, we then obtain \eqref{eq:def_sigma_eff}.
\end{remark}

\begin{remark}\label{rem:HMC}
Another consequence of the fact that  $\Sigma(x,\rho)$ does not depend on $x$ is that a stationary solution of \eqref{eq:Eff} is also the first marginal of a stationary solution of the corresponding effective dynamics for the kinetic Langevin diffusion (with the second marginal, i.e. the distribution of velocities at equilibrium, being the standard Gaussian distribution, and the distribution in the phase space being the product of these two marginals,  see e.g. \cite{guillin2021uniform}). It means that Theorem~\ref{thm:phase_trans_eff} also applies to the kinetic case. Moreover, in this case, in practice, second-order splitting schemes are used, which means that the discretization error is negligible with respect to the stochastic gradient error, and thus the effective dynamics captures the leading term of the numerical errror, see also \cite{monmarche2022hmc} on this topic.
\end{remark}

\noindent Let us sum up the organization of the article.
\begin{itemize}
    \item The Curie-Weiss model is studied in Section~\ref{sec:tout_CW}. We start by recalling the analysis of the phase transition for the classical Curie-Weiss model in Section~\ref{sec:CW_no_minibatch} since the same ideas will be used afterwards. The study of the  Curie-Weiss model with random batches and the proof of Theorem~\ref{thm:RB-CW} are then done in Section~\ref{sec:CW_minibatch},
    \item In Section~\ref{sec:RB_for_IPS} we study the Random Batch Method for interacting particle systems. More specifically we prove  Theorem~\ref{thm:phase_trans_eff} in the specific case of the double-well potential,
    \item Finally, in Appendix~\ref{app:technical} we gather some technical lemmas, and in Appendix~\ref{app:proofs_NL} we prove some results on \eqref{eq:NL} for the double-well potential used in Section~\ref{sec:RB_for_IPS}.
\end{itemize}

%
%
%Section
%
%

\section*{Notation}

\paragraph{For the Curie-Weiss model, with and without random batches:}
\begin{itemize}
    \item $\Omega_N=\{-1,...,1\}^N$ : the set of possible configurations,
    \item $\sigma(n)=(\sigma_1(n),...,\sigma_N(n))$ : the spin configuration at time step $n$,
    \item $\beta$ : the inverse temperature,
    \item $\beta_c$ : the critical inverse temperature,
    \item $H_N$: the Hamiltonian of the Curie-Weiss model given in \eqref{eq:CW_def_H},
    \item $m_N(n)=\frac{1}{N}\sum_{i=1}^N\sigma_i(n)$ : the magnetization at time step $n$,
    \item $r(\sigma,\sigma')$ : transition probability for the Markov chain $(m_N(n))_n$, given in \eqref{eq:CW_def_r}.
    \item $\sigma^{p}(n)=(\sigma^{p}_1(n),...,\sigma^{p}_N(n))$ : the spin configuration of the system with random batches of size $p$ at time step $n$,
    \item $H_{N,p}$ : the Hamiltonian for the system with random batches of size $p$, given in \eqref{eq:RB_CW_def_H},
    \item $m_{N,p}(n)=\frac{1}{N}\sum_{i=1}^N\sigma^p_i(n)$ : the magnetization at time step $n$ for the system with random batches of size $p$,
    \item $r_p(m,m')$ : transition probability for the Markov chain $(m_{N,p}(n))_n$, given in Lemma~\ref{lem:trans_prob}.
    \item $\beta_{c,p}$ : the critical inverse temperature for the system with random batches of size $p$.
\end{itemize}
\paragraph{For the Random Batch Method for interacting particle system :}
\begin{itemize}
    \item $U,W$ : two twice continuously differentiable functions, respectively the confining potential and the interacting potential (see \eqref{eq:IPS}),
    \item $\sigma>0$ : a diffusion coefficient (see \eqref{eq:IPS}),
    \item $(X^i_t)_{i\in\{1,...,N\}}$ : the solution at time $t\in\mathbb{R}^+$ of the interacting particle system \eqref{eq:IPS},
    \item $\bar{X}_t$, $\bar{\rho}_t$ : the solution at time $t\in\mathbb{R}^+$ of the nonlinear limit \eqref{eq:NL} and its law,
    \item $\delta>0$ : a timestep used in the various numerical schemes,
    \item $(X^{i,\delta}_t)_{i\in\{1,...,N\}}$ : the solution at time step $t\in\mathbb{N}$ of the Euler-Maruyama numerical scheme \eqref{eq:D-IPS},
    \item $p\in\mathbb{N}\setminus\{0,1\}$ : the batch size,
    \item $\mathcal{P}_t$ : the partition of $\{1,...,N\}$ at time step $t$ into subsets of size $p$,
    \item $\mathcal{C}^i_t$ : the cluster containing index $i$ at time step $t$ (see \eqref{eq:def_rb_clust}),
    \item $(Y^{i,\delta,p}_t)_{i\in\{1,...,N\}}$ : the solution at time step $t\in\mathbb{N}$ of the numerical scheme with random batches \eqref{eq:D-RB-IPS},
    \item $\bar{Y}^{\delta,p}_t$ : the solution  at time step $t\in\mathbb{N}$ of \eqref{eq:D-RB-NL}, the nonlinear limit of \eqref{eq:D-RB-IPS} as $N\rightarrow\infty$,
    \item $\bar{X}^{e,\delta,p}_t$, $\bar{\rho}^{e,\delta,p}_t$ : the effective dynamics \eqref{eq:Eff} at time $t\in\mathbb{R}^+$ and its law,
    \item $\mu_{\sigma,*}$ for $*\in\{0,\pm\}$, $\sigma_c$ : stationary distributions and critical parameter of \eqref{eq:NL} given in Theorem~\ref{thm:Tugaut},
    \item $\mu^{\delta,p}_{\sigma,*}$ for $*\in\{0,\pm\}$, $\sigma^{eff}_c$ : stationary distributions and critical parameter of \eqref{eq:Eff} given in Theorem~\ref{thm:phase_trans_eff}.
\end{itemize}

%
%
%Section
%
%

\section{Understanding the problem on the Curie-Weiss model}\label{sec:tout_CW}

In order to get a better grasp on the phenomenon we focus on, we begin by studying arguably one of the simplest model admitting a phase transition : the Curie-Weiss model. In Section~\ref{sec:CW_no_minibatch}, we show how we obtain the value of the critical parameter in the classical case. Then, in Section~\ref{sec:CW_minibatch}, we follow the same steps to compute the new critical inverse temperature in the case with random batches.

%
%
%Subsection
%
%

\subsection{...without the Random Batch Method}\label{sec:CW_no_minibatch}

In order to study this critical inverse temperature, we choose to look at the limit of the dynamics with time step $\frac{1}{N}$ as $N$ goes to infinity. $(m_N(n))_n$ is a discrete-time Markov chain with transition operator $U^{(N)}$ given by $U^{(N)}=\left(U^{(N)}_{i,j}\right)_{0\leq i,j\leq N}$ where $U^{(N)}_{i,j}=r\left(-1+\frac{2i}{N},-1+\frac{2j}{N}\right)$. We denote $A_N=N\left(U^{(N)}-I\right)$. We have, for all continuously differentiable functions $f$,
\begin{align*}
A_Nf(m)=&N\frac{1-m}{2}e^{-\beta N\left(\frac{m^2}{2}-\frac{\left(m+\frac{2}{N}\right)^2}{2}\right)_+}\left(f\left(m+\frac{2}{N}\right)-f(m)\right)\\
&+N\frac{1+m}{2}e^{-\beta N\left(\frac{m^2}{2}-\frac{\left(m-\frac{2}{N}\right)^2}{2}\right)_+}\left(f\left(m-\frac{2}{N}\right)-f(m)\right).
\end{align*}
We thus get
\begin{align*}
A_Nf(m)=&N\frac{1-m}{2}e^{-2\beta\left(-m-\frac{1}{N}\right)_+}\left(f\left(m+\frac{2}{N}\right)-f(m)\right)\\
&+N\frac{1+m}{2}e^{-2\beta\left(m+\frac{1}{N}\right)_+}\left(f\left(m-\frac{2}{N}\right)-f(m)\right)\\
=&N\frac{1-m}{2}e^{-2\beta\left(-m-\frac{1}{N}\right)_+}\left(\frac{2}{N}f'(m)+O\left(\frac{1}{N^2}\right)\right)\\
&+N\frac{1+m}{2}e^{-2\beta\left(m+\frac{1}{N}\right)_+}\left(-\frac{2}{N}f'(m)+O\left(\frac{1}{N^2}\right)\right)\\
=&(1-m)e^{-2\beta\left(-m-\frac{1}{N}\right)_+}f'(m)-(1+m)e^{-2\beta\left(m+\frac{1}{N}\right)_+}f'(m)+O\left(\frac{1}{N}\right)\\
&\xrightarrow[N \to \infty]{}  f'(m)\left((1-m)e^{-2\beta(-m)_+}-(1+m)e^{-2\beta m_+}\right),
\end{align*}
which finally yields
\begin{align*}
A_Nf(m)\xrightarrow[N \to \infty]{} 2f'(m)e^{-\beta|m|}\left(\sinh(m\beta)-m\cosh(m\beta)\right).
\end{align*}
By  \cite[Theorem 17.28]{Kal97}, the process $M^{(N)}_t=m_N(\lfloor Nt\rfloor)$ weakly converges to the solution of 
\begin{align*}
\frac{d}{dt}m(t)=2e^{-\beta|m(t)|}\left(\sinh(\beta m(t))-m(t)\cosh(\beta m(t) )\right).
\end{align*}
Denote $f(\beta,m)=2e^{-\beta|m|}\left(\sinh(\beta m)-m\cosh(\beta m )\right)$. We have 
\begin{align*}
f(\beta,m)=0\ \ \ \iff\ \ \ \tanh(\beta m)=m.
\end{align*}
For $\beta>1$, the equation $f(\beta,m)=0$ admits three solutions, and for $\beta\leq1$ only one. Notice that for all $\beta>0$, $f(\beta,0)=0$ : 0 is thus always an equilibrium state for the magnetization. Furthermore
\begin{align*}
\forall \beta>0, \forall m\neq0, \partial_m f(\beta,m)=&-2\beta\text{sign}(m)e^{-\beta|m|}\left(\sinh(\beta m)-m\cosh(\beta m )\right)\\
+&2e^{-\beta|m|}\left((\beta-1)\cosh(\beta m )-\beta m\sinh(\beta m)\right),
\end{align*}
and, extending by continuity, we have $\partial_m f(\beta,0)=2(\beta-1).$
Therefore, for $\beta>1$, $0$ is unstable as $\partial_m f(\beta,0)>0$, and for $\beta\leq1$ it is stable.

Hence a critical inverse temperature $\beta_c=1$, above which there are two stable equilibrium states, and under which there is only one.

%
%
%Subsection
%
%

\subsection{...with the Random Batch Method}\label{sec:CW_minibatch}

To follow the same steps in the case with random batches, we need to compute the transition operator before finding its limit.

%
%
%Subsubsection
%
%

\subsubsection{Transition probabilities}

Let us start by giving explicit values for the transitions probabilities for the magnetization using the Random Batch Method. The proof, which relies on combinatorics arguments, is double-checked via numerical simulations in Figure~\ref{Fig:verif_proba}.

%Lemma

\begin{lemma}\label{lem:trans_prob}
In a system of size $N$, the transition probabilities for the magnetization with random batches of size $p$ are given by
\begin{equation}\label{eq:val_proba_trans_p}
r_p(m,m')=\left\{
\begin{array}{ll}
&\frac{1-m}{2}\binom{N-1}{p-1}^{-1}\sum_{k=0}^{p-1}\binom{\left(\frac{1-m}{2}\right)N-1}{k}\binom{\left(\frac{1+m}{2}\right)N}{p-1-k}e^{-2\beta\left(\frac{2k+1-p}{p}\right)_+}\\
&\hspace{2cm}\text{ if }m'=m+\frac{2}{N}\\
&\frac{1+m}{2}\binom{N-1}{p-1}^{-1}\sum_{k=0}^{p-1}\binom{\left(\frac{1-m}{2}\right)N}{k}\binom{\left(\frac{1+m}{2}\right)N-1}{p-1-k}e^{-2\beta\left(\frac{p-1-2k}{p}\right)_+ }\\
&\hspace{2cm}\text{ if }m'=m-\frac{2}{N}\\
&1-r_p\left(m,m+\frac{2}{N}\right)-r_p\left(m,m-\frac{2}{N}\right)\\
&\hspace{2cm}\text{ if }m'=m\\
&0\hspace{1.8cm}\text{ otherwise.}
\end{array}
\right.
\end{equation}
\end{lemma}

%Proof

\begin{proof}
Notice that, for a given $m$, the number of positive spins is given by $\frac{1+m}{2}N$ and the number of negative spins by $\frac{1-m}{2}N$.
\paragraph{Going right.} Let us calculate the probability of going from $m$ to $m+\frac{2}{N}$. To do so, the chosen spin, denoted $i$, must be of value $-1$, and this will happen with probability $\frac{1-m}{2}$. Then, depending on the cluster $\mathcal{C}$ to which spin $i$ belongs, switching the spin from $-1$ to $+1$ happens with probability
\begin{align*}
\mathbb{P}(\sigma^p_i(n+1)=1&|\sigma^p_i(n)=-1,\mathcal{C})=\exp\left(-\beta\left(-\frac{1}{2p}\sum_{j,l\in\mathcal{C}}\sigma'_j\sigma'_l+\frac{1}{2p}\sum_{j,l\in\mathcal{C}}\sigma^p_j(n)\sigma^p_l(n)\right)_+\right),
\end{align*}
where $\sigma'$ denotes the configuration such that for all $j\neq i$, $\sigma'_j=\sigma^p_j(n)$, and ${\sigma'_i=-\sigma^p_i(n)}$. We have
\begin{align*}
-\frac{1}{2p}\sum_{j,l\in\mathcal{C}}\sigma'_j\sigma'_l+&\frac{1}{2p}\sum_{j,l\in\mathcal{C}}\sigma^p_j(n)\sigma^p_l(n)\\
=&-\frac{1}{2p}\left(\sum_{j,l\in\mathcal{C}, j\neq i,l\neq i}\sigma'_j\sigma'_l-\sum_{j,l\in\mathcal{C}, j\neq i,l\neq i}\sigma^p_j(n)\sigma^p_l(n)+2\sum_{j\in\mathcal{C},j\neq i}\sigma'_j\sigma'_i\right.\\
&\left.\hspace{2cm}-2\sum_{j\in\mathcal{C},j\neq i}\sigma^p_j(n)\sigma^p_i(n)+(\sigma'_i)^2-(\sigma^p_i(n))^2\right)\\
=&-\frac{1}{2p}\left(-2\sigma^p_i(n)\sum_{j\in\mathcal{C},j\neq i}\sigma^p_j(n)-2\sigma^p_i(n)\sum_{j\in\mathcal{C},j\neq i}\sigma^p_j(n)\right)\\
=&\frac{2}{p}\sigma^p_i(n)\sum_{j\in\mathcal{C},j\neq i}\sigma^p_j(n).
\end{align*}
We classify the possible clusters containing $i$ based on the number of negative spins. The number of clusters containing $i$ and $k$ other negative spins is ${\binom{\frac{1-m}{2}N-1}{k}\binom{\frac{1+m}{2}N}{p-1-k}}$ (choosing $k$ spins among the $\frac{1-m}{2}N-1$ negative spins that are not $i$, then the $p-1-k$ spins that remain to construct cluster $\mathcal{C}$ among the positive spins). For $k$ negative spins in cluster $\mathcal{C}$ (without counting $i$), we have
\begin{align*}
\sum_{j\in\mathcal{C},j\neq i}\sigma^p_j(n)=\sum_{j\in\mathcal{C},j\neq i,\sigma^p_j(n)=1}1-\sum_{j\in\mathcal{C},j\neq i, \sigma^p_j(n)=-1}1=p-1-k-k,
\end{align*}
and thus, since $\sigma^p_i(n)=-1$
\begin{align*}
-\frac{1}{2p}\sum_{j,l\in\mathcal{C}}\sigma'_j\sigma'_l+\frac{1}{2p}\sum_{j,l\in\mathcal{C}}\sigma^p_j(n)\sigma^p_l(n)=&2\frac{2k+1-p}{p}
\end{align*}
The total number of possible choices for $\mathcal{C}$ is $\binom{N-1}{p-1}$ (choosing the $(p-1)$ spins that are not $i$). Hence
\begin{align*}
r_p&\left(m,m+\frac{2}{N}\right)=\frac{1-m}{2}\frac{1}{\binom{N-1}{p-1}}\sum_{k=0}^{p-1}\binom{\left(\frac{1-m}{2}\right)N-1}{k}\binom{\left(\frac{1+m}{2}\right)N}{p-1-k}e^{-2\beta\left(\frac{2k+1-p}{p}\right)_+ }
\end{align*}
\paragraph{Going left.}Similar calculations yield the probability of going left : the probability of choosing a spin of value $+1$ is $\frac{1+m}{2}$, then we classify the possible clusters containing this spin based on the number of negative spins.
\end{proof}

\begin{figure}
\centering
\includegraphics[width=\linewidth,,height=0.75\linewidth]{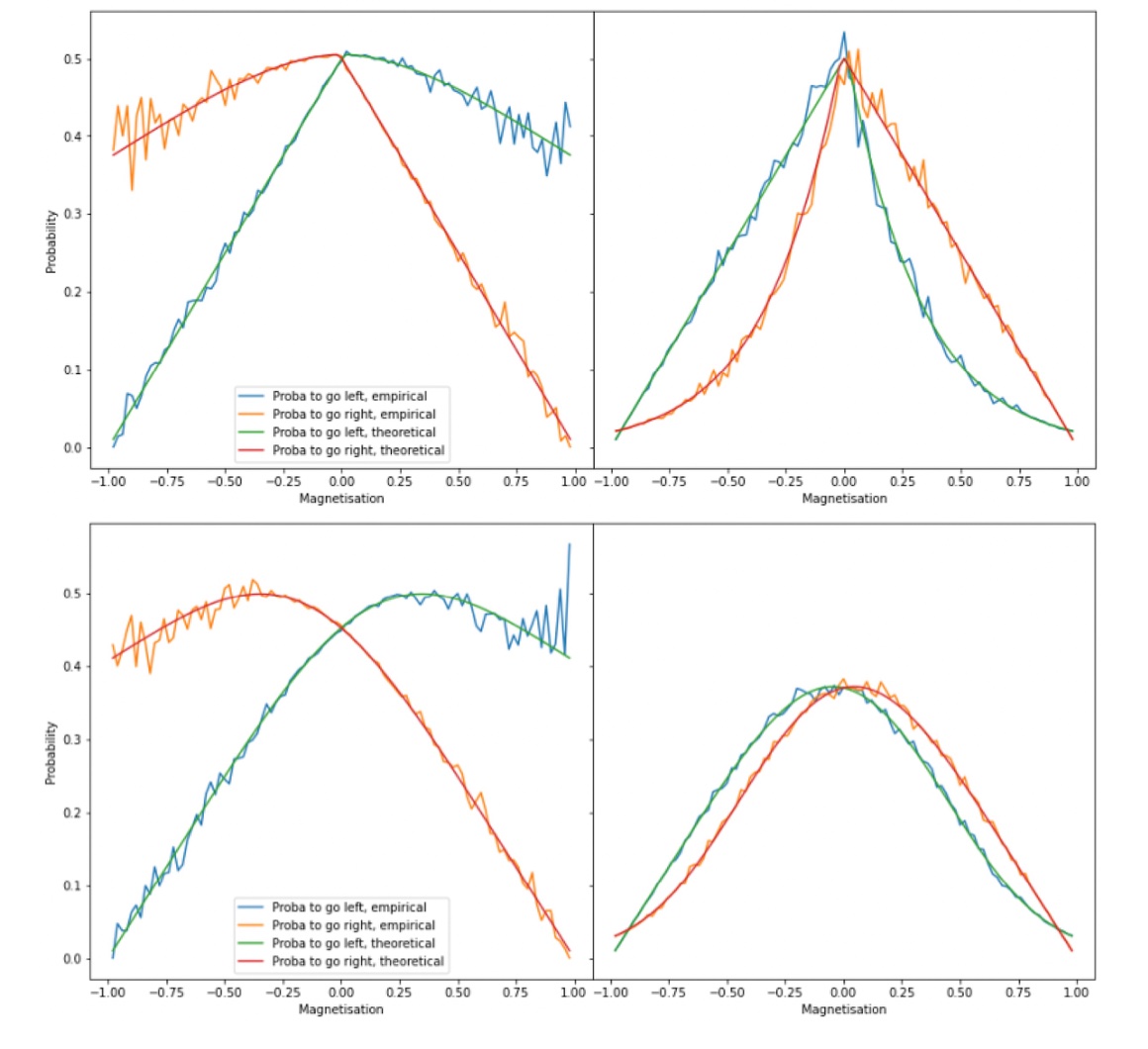}
\caption{Comparison of theoretical and empirical transition probabilities, for $N=100$. The theoretical values are those given in Lemma~\ref{lem:trans_prob}. To numerically compute the empirical transition probabilities, for each initial magnetization in $\{-1,-1+\frac{2}{N},...,1-\frac{2}{N},1\}$, 10 processes are simulated during 1000 timesteps, and we consider the proportion of times the processes go left or right. \textbf{Top : } without random batches. \textbf{Bottom : } with random batches of size $p=10$. \textbf{Left :} for $\beta=0.5$. \textbf{Right :} for $\beta=2$.}
\label{Fig:verif_proba}
\end{figure}

\begin{figure}
\centering
\includegraphics[width=\linewidth]{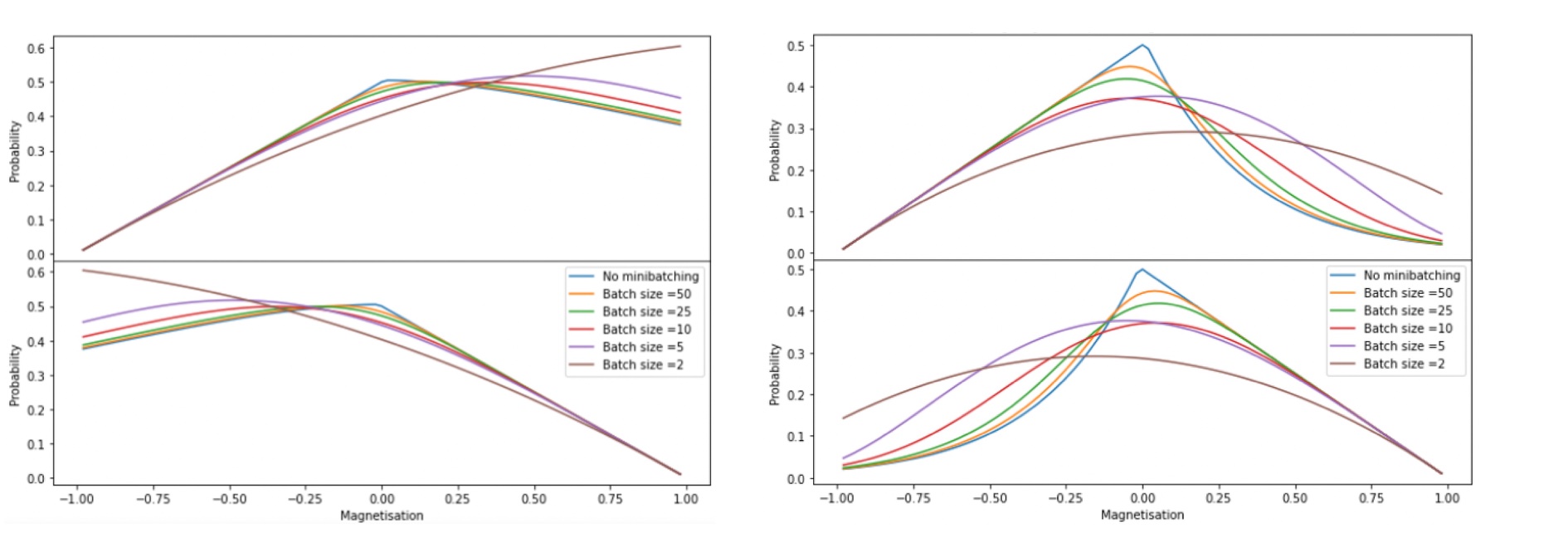}
\caption{Comparison of transition probabilities depending on batch size, for $N=100$. The values given are from Lemma~\ref{lem:trans_prob}. \textbf{Top : } probability of going left. \textbf{Bottom : } probability of going right. \textbf{Left :} for $\beta=0.5$. \textbf{Right :} for $\beta=2$.}
\label{Fig:comp_proba_batch}
\end{figure}

%Remark

\begin{remark}
The values given in \eqref{eq:val_proba_trans_p} are consistent in the case $p=N$. Observe for instance that the only nonzero term in the sum defining $r_N\left(m,m+\frac{2}{N}\right)$ is obtained for $k=\frac{1-m}{2}N-1$. Thus
\begin{align*}
    r_N\left(m,m+\frac{2}{N}\right)=\frac{1-m}{2}e^{-2\beta\left(-m-\frac{1}{N}\right)_+}=r\left(m,m+\frac{2}{N}\right),
\end{align*}
where the value of $r$ is given in \eqref{eq:CW_def_r}.
\end{remark}

%Remark

\begin{remark}
We observe how the transition probabilities evolve with the parameter $p$ in Figure~\ref{Fig:comp_proba_batch}. Furthermore, the values given in \eqref{eq:val_proba_trans_p} allow us to define, on the state space $\{-1,1+\frac{2}{N},...,1-\frac{2}{N},1\}$, a transition matrix for the magnetization. The latter is an irreducible and aperiodic Markov chain on a finite state space, and thus admits a unique invariant measure. We can numerically obtain it by iterating the transition matrix (see Figure~\ref{Fig:comp_gibbs})
\end{remark}

\begin{figure}
\hspace*{-1.2cm}
\includegraphics[width=1.17\linewidth,height= \linewidth, left]{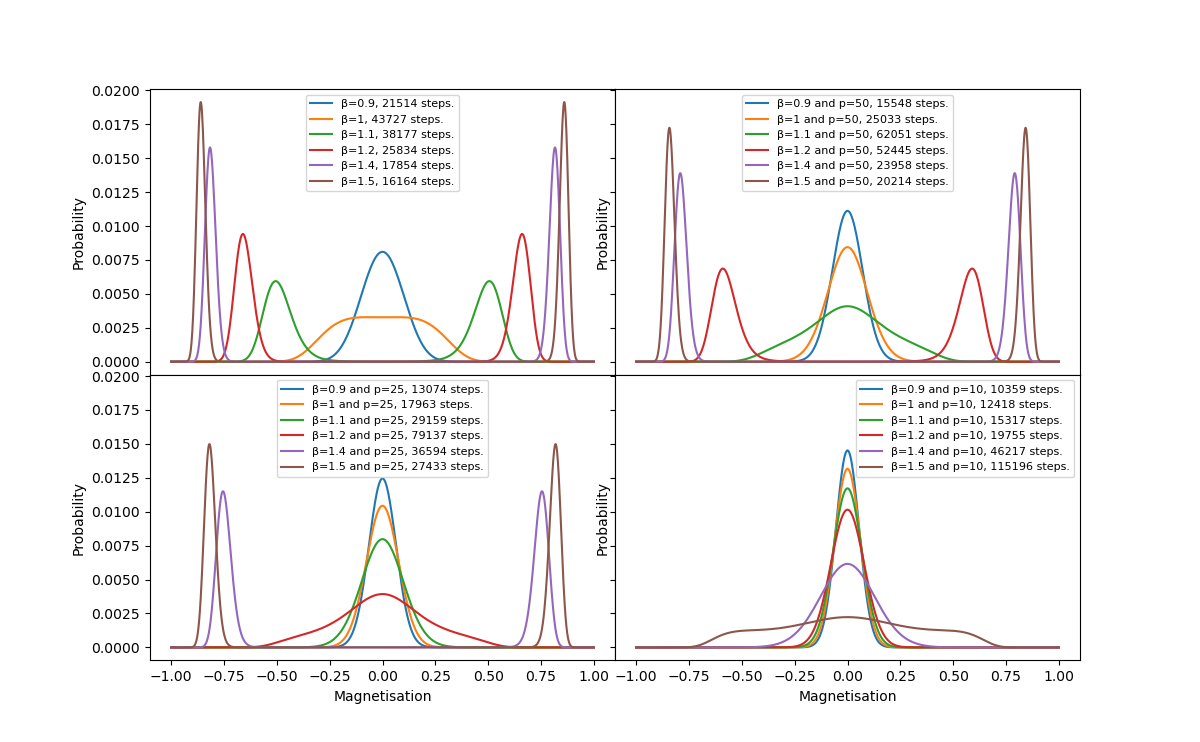}
\vspace*{-1cm}
\caption{Numerical observation of the invariant distribution for the Curie-Weiss model with $N$ spins. Starting from the uniform distribution for the magnetization, we iterate the transition matrix (given by Lemma~\ref{lem:trans_prob}) until the $L^1$ distance between two consecutive iterations is less than a threshold $N\epsilon$, with $N=1000$, $\epsilon=10^{-9}$ and various values for $\beta$. We indicate the number of iterations (or steps) needed before convergence. \textbf{Top {left :}} with no random batches.  \textbf{Top {right :}} with $p=50$. \textbf{Bottom {left :}} with $p=25$. \textbf{Bottom {right :}} with $p=10$.} %test_stat_distrib_transprob_figure(1000,500000,[0.9,1,1.1,1.2,1.4,1.5],[50,25,10],1e-9)
\label{Fig:comp_gibbs}
\end{figure}

%
%
%Subsubsection
%
%

\subsubsection{Study of the critical parameter}

We now wish to show how adding random batches artificially increases the temperature of the system, thus decreasing the critical temperature (or, equivalently, increasing the critical inverse temperature).

\paragraph{Limit ODE.} Let us, like previously, find the limit as $N$ goes to infinity of the dynamics of $\left(m_{N,p}(n)\right)_n$ with time step $\frac{1}{N}$. This discrete-time Markov chain admits a transition operator $U^{(N,p)}$ given by $$U^{(N,p)}=\left(U^{(N,p)}_{i,j}\right)_{0\leq i,j\leq N}\ \ \ \text{ where  }\ \ \ {U^{(N,p)}_{i,j}=r_p\left(-1+\frac{2i}{N},-1+\frac{2j}{N}\right)}.$$ We denote $A^{(p)}_N=N(U^{(N,p)}-I)$ and have, for all continuously differentiable functions $f$,
\begin{align*}
A^{(p)}_Nf(m)=&Nr_p\left(m,m+\frac{2}{N}\right)\left(f(m+\frac{2}{N})-f(m)\right)+Nr_p\left(m,m-\frac{2}{N}\right)\left(f(m-\frac{2}{N})-f(m)\right)\\
=&r_p\left(m,m+\frac{2}{N}\right)\left(2f'(m)+O\left(\frac{1}{N}\right)\right)-r_p\left(m,m-\frac{2}{N}\right)\left(2f'(m)+O\left(\frac{1}{N}\right)\right).
\end{align*}
We have, by standard computations
\begin{align*}
r_p\left(m,m+\frac{2}{N}\right)&=\frac{1-m}{2}\binom{N-1}{p-1}^{-1}\sum_{k=0}^{p-1}\binom{\left(\frac{1-m}{2}\right)N-1}{k}\binom{\left(\frac{1+m}{2}\right)N}{p-1-k}e^{-2\beta\left(\frac{2k+1-p}{p}\right)_+}\\
&\xrightarrow[N \to \infty]{}\frac{1-m}{2}\sum_{k=0}^{p-1}\binom{p-1}{k}\left(\frac{1-m}{2}\right)^k\left(\frac{1+m}{2}\right)^{p-1-k}e^{-2\beta\left(\frac{2k+1-p}{p}\right)_+},
\end{align*}
and likewise
\begin{align*}
r_p\left(m,m-\frac{2}{N}\right)\xrightarrow[N \to \infty]{}\frac{1+m}{2}\sum_{k=0}^{p-1}\binom{p-1}{k}\left(\frac{1-m}{2}\right)^k\left(\frac{1+m}{2}\right)^{p-1-k}e^{-2\beta\left(\frac{p-1-2k}{p}\right)_+}.
\end{align*}
Hence
\begin{align*}
A^{(p)}_Nf(m)\xrightarrow[N \to \infty]{}A^{(p)}f(m),
\end{align*}
where
\begin{align*}
A^{(p)}f(m)&=f'(m)\left(S^{p,\beta}_1(m)-S^{p,\beta}_2(m)\right)-mf'(m)\left(S^{p,\beta}_1(m)+S^{p,\beta}_2(m)\right),\\
S^{p,\beta}_1(m)&=\sum_{k=0}^{p-1}\binom{p-1}{k}\left(\frac{1-m}{2}\right)^k\left(\frac{1+m}{2}\right)^{p-1-k}e^{-2\beta\left(\frac{2k+1-p}{p}\right)_+ }\\
S^{p,\beta}_2(m)&=\sum_{k=0}^{p-1}\binom{p-1}{k}\left(\frac{1-m}{2}\right)^k\left(\frac{1+m}{2}\right)^{p-1-k}e^{-2\beta\left(\frac{p-1-2k}{p}\right)_+ }.
\end{align*}

%Remark

\begin{remark}\label{rem:CW_lim_indep}
Notice that 
\[S^{p,\beta}_1(m)=\mathbb{E}\left(e^{-2\beta\left(\frac{2X_{m,p}+1-p}{p}\right)_+ }\right)\,,\quad  S^{p,\beta}_2(m)=\mathbb{E}\left(e^{-2\beta\left(\frac{p-1-2X_{m,p}}{p}\right)_+ }\right)\,,\]
where $X_{m,p}$ is a random variable following a binomial distribution of parameters $p-1$ and $\frac{1-m}{2}$. Intuitively, for an infinite number of spins, the dynamics of the system relies on the construction of a cluster of size $p$ (containing the chosen spin that may change), which is done by independently taking the remaining $p-1$ spins from an infinite pool containing a proportion  $\frac{1-m}{2}$ of negative spins. 
\end{remark}

Denoting $f_p(\beta,m)=\left(S^{p,\beta}_1(m)-S^{p,\beta}_2(m)\right)-m\left(S^{p,\beta}_1(m)+S^{p,\beta}_2(m)\right)$, by \cite[Theorem 17.28]{Kal97}, the process $M^{(N,p)}_t=m_{N,p}(\lfloor Nt\rfloor)$ weakly converges to the solution of 
\begin{align*}
\frac{d}{dt}m(t)=f_p(\beta,m(t)).
\end{align*}

\paragraph{The cases $p=2$ and $p=3$.} We may directly compute
\begin{align*}
    S^{2,\beta}_1(m)=&\frac{1+m}{2}+\frac{1-m}{2}e^{-\beta},\ \ \ \ \ \  
    S^{2,\beta}_2(m)=\frac{1+m}{2}e^{-\beta}+\frac{1-m}{2},\\
    S^{3,\beta}_1(m)=&\left(\frac{1+m}{2}\right)^2+2\left(\frac{1+m}{2}\right)\left(\frac{1-m}{2}\right)+\left(\frac{1-m}{2}\right)^2e^{-\frac{4\beta}{3}},\\
    S^{3,\beta}_2(m)=&\left(\frac{1+m}{2}\right)^2e^{-\frac{4\beta}{3}}+2\left(\frac{1+m}{2}\right)\left(\frac{1-m}{2}\right)+\left(\frac{1-m}{2}\right)^2,
\end{align*}
which yield
\begin{align*}
    f_2(\beta,m)=&m(1-e^{-\beta})-m(1+e^{-\beta})=-2me^{-\beta},\\
    f_3(\beta,m)=&\left(\left(\frac{1+m}{2}\right)^2-\left(\frac{1-m}{2}\right)^2\right)\left(1-e^{-\frac{4\beta}{3}}\right)\\
    &\hspace{3cm}-m\left(\left(\left(\frac{1+m}{2}\right)^2+\left(\frac{1-m}{2}\right)^2\right)\left(1+e^{-\frac{4\beta}{3}}\right)+(1+m)(1-m)\right)\\
    =&m\left(1-e^{-\frac{4\beta}{3}}\right)-m\left(\frac{(1+m^2)}{2}\left(1+e^{-\frac{4\beta}{3}}\right)+1-m^2\right)\\
    =&-\frac{m}{2}\left(1+3e^{-\frac{4\beta}{3}}\right)+\frac{m^3}{2}\left(1-e^{-\frac{4\beta}{3}}\right).
\end{align*}
For $p=2$ we thus have, $f_2(\beta,m)=0\ \iff\ m=0$, and furthermore notice that $\partial_m f_2(\beta,0)<0,$ which means that 0 is the unique equilibrium state, and it is stable. For $p=3$,
\begin{align*}
    f_3(\beta,m)=0\ \ \ \iff\ \ \ m=0\ \ \ \text{ or }\ \ \ m=\pm\sqrt{\frac{1+3e^{-\frac{4\beta}{3}}}{1-e^{-\frac{4\beta}{3}}}}.
\end{align*}
However, for all $\beta>0$ we have $\sqrt{\frac{1+3e^{-\frac{4\beta}{3}}}{1-e^{-\frac{4\beta}{3}}}}>1$, as well as $\partial_m f_3(\beta,0)=-\frac{1+3e^{-\frac{4\beta}{3}}}{2}<0$. The point 0 is thus the unique equilibrium state, and it is stable.

\begin{figure}
\includegraphics[width=\linewidth,height= 0.7 \linewidth]{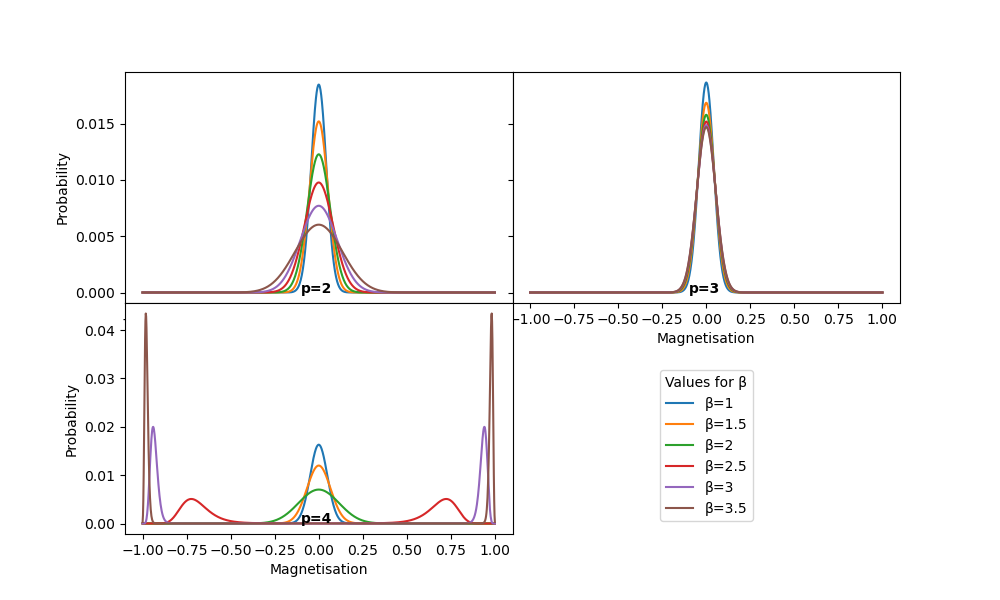}
\vspace*{-1cm}
\caption{Numerical observation of the invariant distribution for the Curie-Weiss model with $N$ spins in the cases $p=2$ (Top right), $p=3$ (Top left) and $p=4$ (Bottom left). Starting from the uniform distribution for the magnetization, we iterate the transition matrix (given in Lemma~\ref{lem:trans_prob}) until the $L^1$ distance between two consecutive iterations is less than a threshold $N\epsilon$, with $N=1000$, $\epsilon=10^{-9}$ and various values for $\beta$.} %test_stat_distrib_transprob_figure_cas_part(1000,500000,[1,1.5,2,2.5,3,3.5],1e-9)
\label{Fig:gibbs_cas_part}
\end{figure}

We may observe this phenomenon in Figure~\ref{Fig:gibbs_cas_part}, in which we compare the cases $p=2$ and $p=3$ with $p=4$.

\paragraph{Existence of a phase transition for $p\geq4$.} First notice that
\begin{align*}
f_p(\beta,0)=&\sum_{k=0}^{p-1}\binom{p-1}{k}\left(\frac{1}{2}\right)^{p-1}e^{-2\beta\left(\frac{2k+1-p}{p}\right)_+}-\sum_{k=0}^{p-1}\binom{p-1}{k}\left(\frac{1}{2}\right)^{p-1}e^{-2\beta\left(\frac{p-1-2k}{p}\right)_+}\\
=&0\ \ \ \text{ by change of variables }\ \ \ k'=p-1-k.
\end{align*}
Thus $m=0$ is for all $\beta>0$ an equilibrium state. The remaining questions, in order to prove Theorem~\ref{thm:RB-CW}, are
\begin{itemize}
    \item is there $\beta_{c,p}>0$ such that for all $\beta<\beta_{c,p}$ we have $\partial_\kappa f(\beta,0)<0$ (in which case $m=0$ is stable) and such that for all $\beta>\beta_{c,p}$ we have $\partial_mf(\beta,0)>0$ (in which case $m=0$ is unstable) ?
    \item do we have $\beta_{c,p}>1$ (in which case the critical temperature has indeed decreased when compared to the case without random batches) ? 
    \item can we give an estimate of $\beta_{c,p}$ ?
\end{itemize}
To answer the first question, we may calculate
\begin{align*}
S'^{p,\beta}_1(m)&=-\sum_{k=0}^{p-1}\frac{k}{2}\binom{p-1}{k}\left(\frac{1-m}{2}\right)^{k-1}\left(\frac{1+m}{2}\right)^{p-1-k}e^{-2\beta\left(\frac{2k+1-p}{p}\right)_+}\\
&+\sum_{k=0}^{p-1}\frac{p-1-k}{2}\binom{p-1}{k}\left(\frac{1-m}{2}\right)^{k}\left(\frac{1+m}{2}\right)^{p-2-k}e^{-2\beta\left(\frac{2k+1-p}{p}\right)_+ }
\end{align*}
and 
\begin{align*}
S'^{p,\beta}_2(m)&=-\sum_{k=0}^{p-1}\frac{k}{2}\binom{p-1}{k}\left(\frac{1-m}{2}\right)^{k-1}\left(\frac{1+m}{2}\right)^{p-1-k}e^{-2\beta\left(\frac{p-1-2k}{p}\right)_+}\\
&+\sum_{k=0}^{p-1}\frac{p-1-k}{2}\binom{p-1}{k}\left(\frac{1-m}{2}\right)^{k}\left(\frac{1+m}{2}\right)^{p-2-k}e^{-2\beta\left(\frac{p-1-2k}{p}\right)_+ },
\end{align*}
which yields
\begin{align*}
\partial_mf_p(\beta,m)=&\left(S'^{p,\beta}_1(m)-S'^{p,\beta}_2(m)\right)-\left(S^{p,\beta}_1(m)+S^{p,\beta}_2(m)\right)-m\left(S'^{p,\beta}_1(m)+S'^{p,\beta}_2(m)\right).
\end{align*}
We thus have
\begin{align*}
\partial_mf_p(\beta,0)&=2S'^{p,\beta}_1(0)-2S^{p,\beta}_1(0)=2\left(\frac{1}{2}\right)^{p-1}\sum_{k=0}^{p-1}(p-2-2k)\binom{p-1}{k}e^{-2\beta\left(\frac{2k+1-p}{p}\right)_+}.
\end{align*}
First, notice
\begin{align*}
\partial_\beta(\partial_mf_p(\beta,0))
=&-4\left(\frac{1}{2}\right)^{p-1}\sum_{k=0}^{p-1}(p-2-2k)\left(\frac{2k+1-p}{p}\right)_+\binom{p-1}{k}e^{-2\beta\left(\frac{2k+1-p}{p}\right)_+}>0.
\end{align*}
The function $\beta\mapsto\partial_mf_p(\beta,0)$ is therefore an increasing function, which furthermore satisfies \linebreak${\partial_mf_p(0,0)<0}$ and ${\lim_{\beta\to\infty}\partial_mf_p(\beta,0)>0}$, hence a unique critical parameter $\beta_{c,p}>0$.

%Remark

\begin{remark}
We use the assumption $p\geq4$ in order to prove $\lim_{\beta\to\infty}\partial_mf_p(\beta,0)>0$. Indeed
\begin{align*}
    \lim_{\beta\to\infty}\partial_mf_p(\beta,0)=2\left(\frac{1}{2}\right)^{p-1}\sum_{k=0}^{p-1}(p-2-2k)\binom{p-1}{k}\mathds{1}_{k\leq\frac{p-1}{2}}.
\end{align*}
If $p$ is even, all the terms in the sum are nonnegative, and if $p\geq4$, at least one term is positive. If $p$ is odd, one term is negative, and if $p\geq5$ it can easily be shown that it is compensated by the positive terms.
\end{remark}

\paragraph{Estimation of the critical parameter.} Denoting $X_p$ a random variable following a binomial distribution of parameters $p-1$ and $\frac{1}{2}$, we have
\begin{align*}
\partial_mf_p(\beta,0)=2\mathbb{E}\left((p-2-2X_p)e^{-2\beta\left(\frac{2X_p+1-p}{p}\right)_+ }\right):=g_{p}(\beta).
\end{align*}
We are thus looking for the unique $\beta_{c,p}>0$ such that $g_{p}(\beta_{c,p})=0$.

Let $Y_p=2\frac{X_p}{p}-\frac{p-1}{p}$. We have
\begin{equation}\label{eq:def_g_MC}
g_{p}(\beta)=\mathbb{E}\left(2(-pY_p-1)e^{-2\beta (Y_p)_+}\right).
\end{equation}
Since $X_p$ and $p-1-X_p$ have the same law, $Y_p$ has the same law as $2\frac{p-1-X_p}{p}-\frac{p-1}{p}=\frac{p-1}{p}-2\frac{X_p}{p}=-Y_p$. Thus
\begin{align*}
g_{p}(\beta)=&-\mathbb{E}\left(pY_pe^{-2\beta (Y_p)_+}\right)-\mathbb{E}\left(p(-Y_p)e^{-2\beta (-Y_p)_+}\right)-\mathbb{E}\left(e^{-2\beta (Y_p)_+}\right)-\mathbb{E}\left(e^{-2\beta (-Y_p)_+}\right)\\
=&-\mathbb{E}\left(pY_p\left(e^{-2\beta (Y_p)_+}-e^{-2\beta (-Y_p)_+}\right)\right)-\mathbb{E}\left(e^{-2\beta (Y_p)_+}+e^{-2\beta (-Y_p)_+}\right)\\
=&\mathbb{E}\left(2pY_pe^{-\beta|Y_p|}\sinh(\beta Y_p)\right)-\mathbb{E}\left(2e^{-\beta|Y_p|}\cosh(\beta Y_p)\right)\\
=&2\mathbb{E}\left(\cosh(\beta Y_p)e^{-\beta|Y_p|}\left(pY_p\tanh(\beta Y_p)-1\right)\right).
\end{align*}
As this is an increasing function in $\beta$, in order to prove that $\beta_{c,p}>1$, it is sufficient to prove that $g_{p}(1)<0$. The Law of Large Number and the Central Limit Theorem yield
\begin{align*}
Y_p\xrightarrow[p \to \infty]{a.s}0\ \ \ \text{ and }\ \ \ \frac{p}{\sqrt{p-1}}Y_p\xrightarrow[p \to \infty]{law}\mathcal{N}(0,1).
\end{align*}
We have
\begin{align*}
g_{p}(\beta)=&2\mathbb{E}\left(\left(1+\frac{\beta^2Y_p^2}{2}+o(Y_p^2)\right)\left(1-\beta|Y_p|+\frac{\beta^2Y_p^2}{2}+o(Y_p^2)\right)\right.\\
&\left.\hspace{5cm}\times\left(pY_p\left(\beta Y_p-\frac{\beta^3Y_p^3}{3}+o(Y_p^3)\right)-1\right)\right)\\
=&2\mathbb{E}\left(\left(\frac{p-1}{p}\beta \frac{p^2}{p-1}Y_p^2-1\right)-\beta|Y_p|\left(\frac{p-1}{p}\beta \frac{p^2}{p-1}Y_p^2-1\right)+O(Y_p^2)+O(pY_p^4)\right)\\
=&2\left(\beta\frac{p-1}{p}\mathbb{E}\left( \frac{p^2}{p-1}Y_p^2\right)-1-\frac{(p-1)^{3/2}}{p^2}\left(\beta^2\mathbb{E}\left(\left|\frac{p}{\sqrt{p-1}}Y_p\right|^3\right)-\beta\mathbb{E}\left(\left|\frac{p}{\sqrt{p-1}}Y_p\right|\right)\right)\right.\\
&\left.\hspace{2cm}+\frac{1}{p}\mathbb{E}\left(O(pY_p^2)+O(p^2Y_p^4)\right)\right)\\
=&2\left(\beta\frac{p-1}{p}\mathbb{E}\left( Z^2\right)-1-\frac{(p-1)^{3/2}}{p^2}\left(\beta^2\left(\mathbb{E}\left|Z\right|^3+o\left(1\right)\right)-\beta\left(\mathbb{E}\left(\left|Z\right|\right)+o(1)\right)\right)+O\left(\frac{1}{p}\right)\right)\\
=&2(\beta-1)-\frac{2}{\sqrt{p}}\sqrt{\frac{2}{\pi}}\left(2\beta^2-\beta\right)+o\left(\frac{1}{\sqrt{p}}\right),
\end{align*}
where for this last equality, we use Lemma~\ref{lem:cv_mom} and the fact that, for $Z\sim\mathcal{N}(0,1)$, $\mathbb{E}|Z|=\sqrt{\frac{2}{\pi}}$ and $\mathbb{E}(|Z|^3)=2\sqrt{\frac{2}{\pi}}$.

In the end, we obtain, again, the fact that $g_p(1)\xrightarrow[p \to \infty]{}0$ (hence the correct critical parameter at the limit) and the fact that, at least for $p$ sufficiently large, $g_p(1)<0$. For smaller values of $p$, we rely on numerical simulations to verify $g_p(1)<0$ (See Figure~\ref{Fig:Approx_gp1_et_gpbc}).  Let us find an approximation of $\beta_{c,p}$ by using the fact that $g_{p}(\beta_{c,p})=0$. We have
\begin{align*}
    2\sqrt{\frac{2}{p\pi}}\beta_{c,p}^2-\left(1+\sqrt{\frac{2}{p\pi}}\right)\beta_{c,p}+1+o\left(\frac{1}{\sqrt{p}}\right)=0,
\end{align*}
i.e
\begin{align*}
    \beta_{c,p,\pm}=&\frac{1}{4}\sqrt{\frac{p\pi}{2}}\left(\left(1+\sqrt{\frac{2}{p\pi}}\right)\pm\left(\left(1+\sqrt{\frac{2}{p\pi}}\right)^2-8\sqrt{\frac{2}{p\pi}}\left(1+o\left(\frac{1}{\sqrt{p}}\right)\right)\right)^{1/2}\right)\\
    =&\frac{1}{4}\sqrt{\frac{p\pi}{2}}\left(1+\sqrt{\frac{2}{p\pi}}\pm\left(1+\frac{2}{p\pi}-6\sqrt{\frac{2}{p\pi}}+o\left(\frac{1}{p}\right)\right)^{1/2}\right)\\
    =&\frac{1}{4}\sqrt{\frac{p\pi}{2}}\left(1+\sqrt{\frac{2}{p\pi}}\pm\left(1+\frac{1}{p\pi}-3\sqrt{\frac{2}{p\pi}}-\frac{9}{p\pi}+o\left(\frac{1}{p}\right)\right)\right),
\end{align*}
thus
\begin{align*}
    \beta_{c,p}=&\frac{1}{4}\sqrt{\frac{p\pi}{2}}\left(1+\sqrt{\frac{2}{p\pi}}-\left(1-3\sqrt{\frac{2}{p\pi}}-\frac{8}{p\pi}+o\left(\frac{1}{p}\right)\right)\right)\\
    =&\frac{1}{4}\sqrt{\frac{p\pi}{2}}\left(4\sqrt{\frac{2}{p\pi}}+\frac{8}{p\pi}+o\left(\frac{1}{p}\right)\right)\\
    =&1+\sqrt{\frac{2}{p\pi}}+o\left(\frac{1}{\sqrt{p}}\right).
\end{align*}
We have thus proved Theorem~\ref{thm:RB-CW}.

\begin{figure}
\centering
\includegraphics[width=0.7\linewidth]{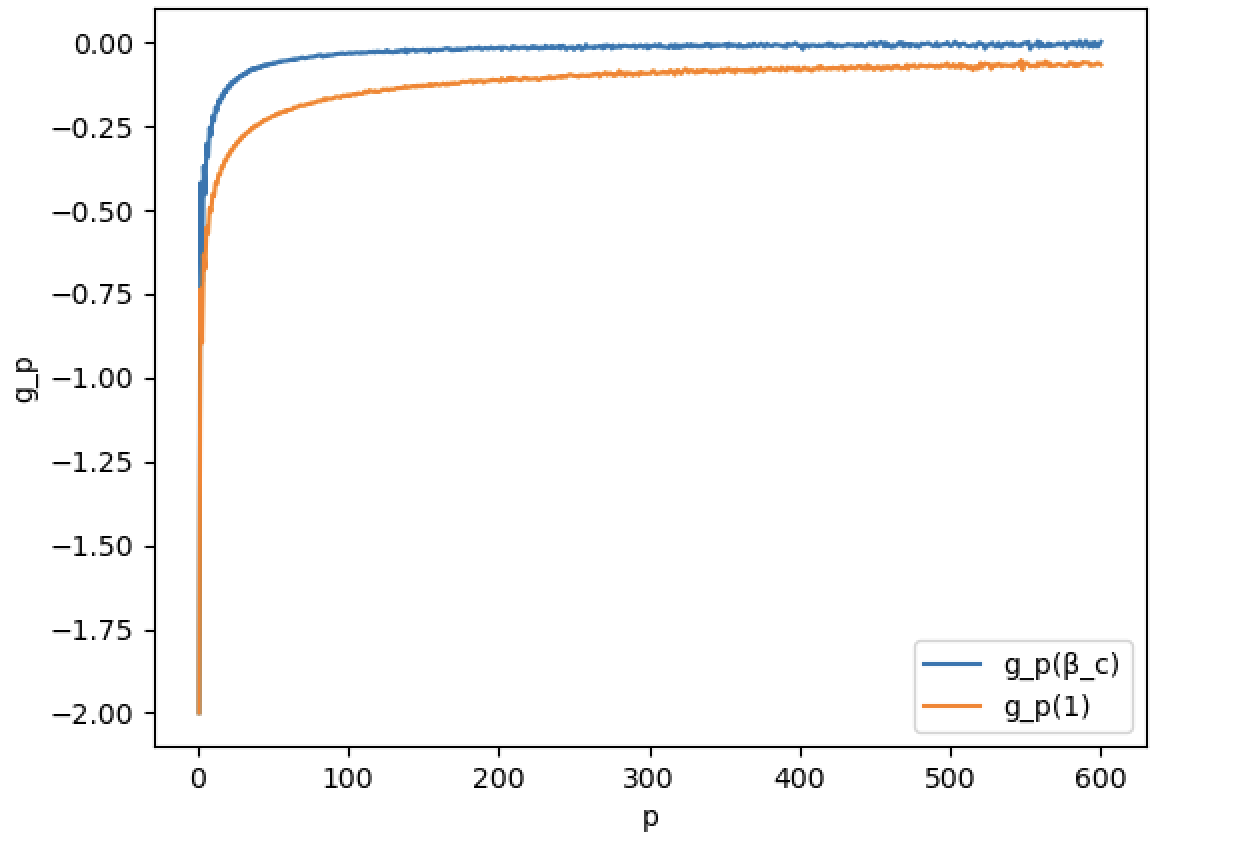}
\caption{Numerical values for $g_p(1)$ and $g_p\left(\tilde \beta_c\right)$, with $\tilde \beta_c =  1+\sqrt{\frac{2}{p\pi}}$. The value of $g_p$ given in \eqref{eq:def_g_MC} is computed via Monte-Carlo approximation using $M=10^8$ samples for $Y_p$.}
\label{Fig:Approx_gp1_et_gpbc}
\end{figure}

%
%
%Section
%
%

\section{Random Batch Method for interacting particle systems and stationary distribution(s)}\label{sec:RB_for_IPS}

We now turn our attention to the study of \eqref{eq:D-RB-IPS} for a given batch size $p\in\mathbb{N}\setminus\{0,1\}$. In the specific case of $U$ and $W$ given in \eqref{eq:def_U_W}, we study the phase transition for \eqref{eq:Eff} and prove Theorem~\ref{thm:phase_trans_eff}. The stationary distributions of \eqref{eq:NL}, provided there exists one, are defined by the solutions of
\begin{equation}\label{eq:def_stat_nl}
    \mu_\sigma(dx)=\frac{\exp\left(-\frac{1}{\sigma}\left(U(x)+W\ast\mu_\sigma(x)\right)\right)}{\int \exp\left(-\frac{1}{\sigma}\left(U(y)+W\ast\mu_\sigma(y)\right)\right)dy}dx.
\end{equation}
We consider the case of linear interactions in a double well potential in dimension one, i.e $U$ and $W$ given in \eqref{eq:def_U_W}, which in particular implies
\begin{align*}
    \Sigma(x,\rho)=&(\nabla W)^2\ast\rho(x)-(\nabla W\ast\rho(x))^2
    =L_W^2\left(\int y^2\rho(dy)-\left(\int y\rho(dy)\right)^2\right)
    =L_W^2\text{Var}(\rho).
\end{align*}
Denote, for a measure $\mu$, 
\begin{align*}
    \kappa_1(\mu)=\int_\mathbb{R}x\mu(dx)\ \ \ \text{ and }\ \ \ \kappa_2(\mu)=\int_\mathbb{R}(x-\kappa_1(\mu))^2\mu(dx).
\end{align*}
The stationary distributions of \eqref{eq:Eff}, provided there exist one, are thus similarly defined by the solutions of
\begin{equation}\label{eq:def_stat_eff}
    \mu^{\delta,p}_\sigma(dx)=\frac{\exp\left(-\frac{2}{2\sigma+\frac{\delta L_W^2}{p-1}\kappa_2\left(\mu^{\delta,p}_\sigma\right)}\left(U(x)+\frac{L_W}{2}\left|x-\kappa_1\left(\mu^{\delta,p}_\sigma\right)\right|^2\right)\right)}{\int \exp\left(-\frac{2}{2\sigma+\frac{\delta L_W^2}{p-1}\kappa_2\left(\mu^{\delta,p}_\sigma\right)}\left(U(y)+\frac{L_W}{2}\left|y-\kappa_1\left(\mu^{\delta,p}_\sigma\right)\right|^2\right)\right)dy}dx.
\end{equation}
The pair $(\kappa_1(\mu^{\delta,p}_\sigma),\kappa_2(\mu^{\delta,p}_\sigma))$ is therefore a solution of 
\begin{align}
    \kappa_1=&\frac{\int_\mathbb{R}x\exp\left(-\frac{2}{2\sigma+\frac{\delta L_W^2}{p-1}\kappa_2}\left(U(x)+\frac{L_W}{2}|x-\kappa_1|^2\right)\right)dx}{\int_\mathbb{R} \exp\left(-\frac{2}{2\sigma+\frac{\delta L_W^2}{p-1}\kappa_2}\left(U(x)+\frac{L_W}{2}|x-\kappa_1|^2\right)\right)dx}\label{eq:pt_fixe_1}\\
    \kappa_2=&\frac{\int_\mathbb{R}\left(x-\kappa_1\right)^2\exp\left(-\frac{2}{2\sigma+\frac{\delta L_W^2}{p-1}\kappa_2}\left(U(x)+\frac{L_W}{2}|x-\kappa_1|^2\right)\right)dx}{\int_\mathbb{R} \exp\left(-\frac{2}{2\sigma+\frac{\delta L_W^2}{p-1}\kappa_2}\left(U(x)+\frac{L_W}{2}|x-\kappa_1|^2\right)\right)dx}\label{eq:pt_fixe_2}.
\end{align}
Thanks to \eqref{eq:def_stat_eff}, solving for $(\kappa_1, \kappa_2)$ the system of equations \eqref{eq:pt_fixe_1}-\eqref{eq:pt_fixe_2} is equivalent to finding a stationary distribution of \eqref{eq:Eff}. Define
\begin{align}
    g(x,\sigma,\kappa)=&\exp\left(-\frac{1}{\sigma}\left(U(x)+\frac{L_W}{2}|x-\kappa|^2\right)\right)\label{eq:def_g},\\
    f_1(\sigma,\kappa)=&\frac{\int_\mathbb{R}xg(x,\sigma,\kappa)dx}{\int_\mathbb{R} g(x,\sigma,\kappa)dx}\label{eq:def_f1},\\
    f_2(\sigma,\kappa)=&\frac{\int_\mathbb{R}(x-\kappa)^2g(x,\sigma,\kappa)dx}{\int_\mathbb{R}g(x,\sigma,\kappa)dx}\label{eq:def_f2},
\end{align}
such that, for the symbol $*\in\{0,\pm\}$ and $\mu_{\sigma,*}$ defined in Theorem~\ref{thm:Tugaut},  $\kappa_1(\mu_{\sigma,*})$ is a solution of $\kappa_1(\mu_{\sigma,*})=f_1(\sigma,\kappa_1(\mu_{\sigma,*}))$ and $\kappa_2(\mu_{\sigma,*})=f_2(\sigma,\kappa_1(\mu_{\sigma,*}))$  is the corresponding variance. 

\plb{
\begin{remark}[Wellposedness]
Since we are looking for stationary distributions of \eqref{eq:Eff}, we are in reality considering specific solutions with constant mean and variance. We could thus choose not to worry about wellposedness of \eqref{eq:Eff} since we could technically forget the nonlinearity.

However, for the sake of completeness, let us quickly turn our attention to the existence and uniqueness of solutions of \eqref{eq:Eff} in the specific case of $U$ and $W$ given in Theorem~\ref{thm:phase_trans_eff}, for which the process is
\begin{align*}
    \left\{
    \begin{array}{l}
    d\bar{X}^{e,\delta,p}_t=u(\bar{X}^{e,\delta,p}_t)+f\ast \rho^{e,\delta,p}_t(\bar{X}^{e,\delta,p}_t)dt+\bar{\sigma}(\rho^{e,\delta,p}_t)dB_t,\\
    \rho^{e,\delta,p}_t=\text{Law}(\bar{X}^{e,\delta,p}_t)
    \end{array}
    \right.
\end{align*}
in dimension 1, with
\begin{align*}
    u(x)=-x^3+x,\ \ \ f(x)=-L_Wx,\ \ \ 
    \text{and }\ \ \ \bar{\sigma}(\mu)=\left(2\sigma+\frac{\delta L_W^2}{p-1}\text{Var}(\mu)\right)^{1/2}.
\end{align*}
To prove existence and uniqueness, we use \cite[Theorem~2.5]{CdRS23}. To do so, we notice that there is a constant $C>0$ (which may change from line to line) such that
\begin{align*}
    (x-y)(u(x)-u(y))\leq C|x-y|^2,\ \ \ |u(x)-u(y)|\leq C(1+x^2+y^2)|x-y|.
\end{align*}
Another assumption we have to verify, and which is the main technical difficulty, is the following
\begin{align*}
    \left|\bar{\sigma}(\mu)-\bar{\sigma}(\nu)\right|\leq C\mathcal{W}_2(\mu,\nu).
\end{align*}
Since the square root function is Lipschitz continuous on $[2\sigma,\infty[$, it is sufficient to prove that there is $C>0$ such that $\left|\text{Var}(\mu)-\text{Var}(\nu)\right|\leq C\mathcal{W}_2(\mu,\nu).$
Denoting $X$ (resp. $Y$) a random variable distributed according to $\mu$ (resp. $\nu$), we have
\begin{align*}
    \left|\text{Var}(\mu)-\text{Var}(\nu)\right|\leq& \left|\mathbb{E}(X^2)-\mathbb{E}(Y^2)\right|+\left|\mathbb{E}(X)^2-\mathbb{E}(Y)^2\right|\\
    \leq&\left|\mathbb{E}((X-Y)(X+Y))\right|+\left|\mathbb{E}(X-Y)\mathbb{E}(X+Y)\right|\\
    \leq& 2\sqrt{\mathbb{E}((X-Y)^2)}\sqrt{\mathbb{E}((X+Y)^2)},
\end{align*}
where we use for this last line Cauchy-Schwarz inequality and Jensen's inequality. Considering $(X,Y)$ to be distributed according to the optimal coupling for the Wasserstein-2 distance between $\mu$ and $\nu$, we thus have
\begin{align*}
    \left|\text{Var}(\mu)-\text{Var}(\nu)\right|\leq C\sqrt{\mathbb{E}(X^2)+\mathbb{E}(Y^2)} \mathcal{W}_2(\mu,\nu).
\end{align*}
We thus should obtain existence and uniqueness for processes with bounded second moments. 

Let us now use this idea to prove wellposedness. We consider the following process, for a given $K>0$
\begin{equation}\label{eq:process_tronque}
    \left\{
    \begin{array}{l}
    d\bar{X}^{K}_t=u(\bar{X}^{K}_t)+f\ast \rho^{K}_t(\bar{X}^{K}_t)dt+\bar{\sigma}_K(\rho^{K}_t)dB_t,\\
    \rho^{K}_t=\text{Law}(\bar{X}^{K}_t)
    \end{array}
    \right.
\end{equation}
where $\bar{\sigma}_K(\mu)=\left(2\sigma+\frac{\delta L_W^2}{p-1}\left(\int x^2 d\mu(x) \wedge K-\left(\int xd\mu(x)\right)^2\wedge K\right)\right)^{1/2}.$ 
Let $a,b\in\mathbb{R}$ and $c\geq0$, we have $\left|a^2\wedge c-b^2\wedge c\right|\leq \left|a-b\right|\left|\sqrt{a^2\wedge c}+\sqrt{b^2\wedge c}\right|$. We obtain similarly as before, for $(X,Y)$ distributed according to the optimal coupling of $\mu$ and $\nu$
\begin{align*}
    \left|\mathbb{E}(X^2) \wedge K\right.&\left.-\mathbb{E}(X)^2\wedge K-\mathbb{E}(Y^2) \wedge K+\mathbb{E}(Y)^2\wedge K\right|\\
    \leq&\left|\mathbb{E}(X^2) \wedge K-\mathbb{E}(Y^2) \wedge K\right|+\left|\mathbb{E}(X)^2\wedge K-\mathbb{E}(Y)^2\wedge K\right|\\
    \leq&2\sqrt{K}\left|\sqrt{\mathbb{E}(X^2)}-\sqrt{\mathbb{E}(Y^2)}\right|+2\sqrt{K}\pierre{\left| |\mathbb{E}(X)|-|\mathbb{E}(Y)|\right|.}
\end{align*}
On one hand we have by Cauchy-Schwarz inequality
\begin{align*}
    \mathbb{E}((X-Y)^2)-\left|\sqrt{\mathbb{E}(X^2)}-\sqrt{\mathbb{E}(Y^2)}\right|^2=2\sqrt{\mathbb{E}X^2}\sqrt{\mathbb{E}Y^2}-2\mathbb{E}(XY)\geq0,
\end{align*}
which implies
\begin{align*}
    \left|\sqrt{\mathbb{E}(X^2)}-\sqrt{\mathbb{E}(Y^2)}\right|\leq& \sqrt{\mathbb{E}((X-Y)^2)}=\mathcal{W}_2(\mu,\nu).
\end{align*}
On the other hand
\begin{align*}
    \pierre{\left| |\mathbb{E}(X)|-|\mathbb{E}(Y)|\right|}\leq&\left|\mathbb{E}(X)-\mathbb{E}(Y)\right|\leq \mathbb{E}|X-Y|\leq \sqrt{\mathbb{E}((X-Y)^2)}=\mathcal{W}_2(\mu,\nu).
\end{align*}
Hence, $\bar{\sigma}_K$ is Lipschitz continuous. 

\pierre{Thus,  \cite[Theorem~2.5]{CdRS23} applies and} we have strong existence and uniqueness for \eqref{eq:process_tronque}. Furthermore, up until the stopping time 
\begin{align*}
    T_K=\inf\left\{t\geq0\text{ s.t. }\int x^2 d\rho^K_t(x)>K\right\},
\end{align*}
the solution of \eqref{eq:process_tronque} coincides with the solution of \eqref{eq:Eff}. Then, applying Itô's formula on the process \eqref{eq:Eff} for the function $x\mapsto x^2$, we obtain that the second moment of the solution is bounded. Hence, for $K$ large enough depending on the initial condition, $T_K=\infty$. We finally obtain the wellposedness of \eqref{eq:Eff}.
\end{remark}
}

Our goal is to compare the stationary distribution(s) for \eqref{eq:Eff} to the stationary distribution(s) for \eqref{eq:NL}, in particular in regards to this critical parameter $\sigma_c$. To do so, we begin by showing that any stationary distribution of \eqref{eq:NL} is a stationary distribution of \eqref{eq:Eff}, and conversely.

%Lemma

\begin{lemma}\label{lem:equiv_stat_NL_Eff}
Let $\mu$ be a probability measure on $\mathbb{R}$. 
\begin{itemize}
    \item If $\mu$ is a solution of \eqref{eq:def_stat_eff} for a diffusion coefficient $\sigma'$, then $\mu$ is a solution of \eqref{eq:def_stat_nl} for a diffusion coefficient $\sigma=\sigma'+\frac{\delta L_W^2}{2(p-1)}\kappa_2(\mu)$,
    \item If $\mu$ is a solution of \eqref{eq:def_stat_nl} for a diffusion coefficient $\sigma$ and $\frac{\delta}{p-1}<\frac{2\sigma}{ L_W^2\kappa_2(\mu)}$, then $\mu$ is a solution of \eqref{eq:def_stat_eff} for a diffusion coefficient $\sigma'=\sigma-\frac{\delta L_W^2}{2(p-1)}\kappa_2(\mu)$.
\end{itemize}
\end{lemma}

%Proof

\begin{proof}
Let us prove the two points.

\noindent\textbf{A stationary distribution of \eqref{eq:Eff} is a stationary distribution of \eqref{eq:NL}.} Assume $\mu$ satisfies \eqref{eq:def_stat_eff} for a given $(\delta,p,\sigma')$, which in particular is equivalent to the pair $(\kappa_1(\mu),\kappa_2(\mu))$ satisfying
\begin{align*}
    \kappa_1(\mu)=&\frac{\int_\mathbb{R}x\exp\left(-\frac{2}{2\sigma'+\frac{\delta L_W^2}{p-1}\kappa_2(\mu)}\left(U(x)+\frac{L_W}{2}|x-\kappa_1(\mu)|^2\right)\right)dx}{\int_\mathbb{R} \exp\left(-\frac{2}{2\sigma'+\frac{\delta L_W^2}{p-1}\kappa_2(\mu)}\left(U(x)+\frac{L_W}{2}|x-\kappa_1(\mu)|^2\right)\right)dx}\\
    \kappa_2(\mu)=&\frac{\int_\mathbb{R}\left(x-\kappa_1(\mu)\right)^2\exp\left(-\frac{2}{2\sigma'+\frac{\delta L_W^2}{p-1}\kappa_2(\mu)}\left(U(x)+\frac{L_W}{2}|x-\kappa_1(\mu)|^2\right)\right)dx}{\int_\mathbb{R} \exp\left(-\frac{2}{2\sigma'+\frac{\delta L_W^2}{p-1}\kappa_2(\mu)}\left(U(x)+\frac{L_W}{2}|x-\kappa_1(\mu)|^2\right)\right)dx}.
\end{align*}
Denoting $\sigma=\sigma'+\frac{\delta L_W^2}{2(p-1)}\kappa_2(\mu)$, we thus have
\begin{align*}
    \kappa_1(\mu)=&\frac{\int_\mathbb{R}x\exp\left(-\frac{1}{\sigma}\left(U(x)+\frac{L_W}{2}|x-\kappa_1(\mu)|^2\right)\right)dx}{\int_\mathbb{R} \exp\left(-\frac{1}{\sigma}\left(U(x)+\frac{L_W}{2}|x-\kappa_1(\mu)|^2\right)\right)dx}\\
    \kappa_2(\mu)=&\frac{\int_\mathbb{R}\left(x-\kappa_1(\mu)\right)^2\exp\left(-\frac{1}{\sigma}\left(U(x)+\frac{L_W}{2}|x-\kappa_1(\mu)|^2\right)\right)dx}{\int_\mathbb{R} \exp\left(-\frac{1}{\sigma}\left(U(x)+\frac{L_W}{2}|x-\kappa_1(\mu)|^2\right)\right)dx}.
\end{align*}
Thus, if $\mu$ is a stationary distribution for \eqref{eq:Eff} with diffusion coefficient $\sigma'$ and parameters $\delta$ and $p$, it is also a stationary distribution for \eqref{eq:NL} with diffusion coefficient $\sigma=\sigma'+\frac{\delta L_W^2}{2(p-1)}\kappa_2(\mu)$.\\

\noindent\textbf{A stationary distribution of \eqref{eq:NL} is a stationary distribution of \eqref{eq:Eff}.}
Assume $\mu$ satisfies \eqref{eq:def_stat_eff} for a given $\sigma$, which in particular is equivalent to the pair $(\kappa_1(\mu),\kappa_2(\mu))$ satisfying
\begin{align*}
    \kappa_1(\mu)=&\frac{\int_\mathbb{R}x\exp\left(-\frac{1}{\sigma}\left(U(x)+\frac{L_W}{2}|x-\kappa_1(\mu)|^2\right)\right)dx}{\int_\mathbb{R} \exp\left(-\frac{1}{\sigma}\left(U(x)+\frac{L_W}{2}|x-\kappa_1(\mu)|^2\right)\right)dx}\\
    \kappa_2(\mu)=&\frac{\int_\mathbb{R}\left(x-\kappa_1(\mu)\right)^2\exp\left(-\frac{1}{\sigma}\left(U(x)+\frac{L_W}{2}|x-\kappa_1(\mu)|^2\right)\right)dx}{\int_\mathbb{R} \exp\left(-\frac{1}{\sigma}\left(U(x)+\frac{L_W}{2}|x-\kappa_1(\mu)|^2\right)\right)dx}.
\end{align*}
Consider parameters $\delta$ and $p$, and denote $\sigma'=\sigma-\frac{\delta L_W^2}{2(p-1)}\kappa_2(\mu)$. Notice $\kappa_2(\mu)$ is independent of $\delta$ and $p$ thus, provided $\frac{\delta }{(p-1)}$ is small enough, we may ensure $\sigma'>0$ and have
\begin{align*}
    \kappa_1(\mu)=&\frac{\int_\mathbb{R}x\exp\left(-\frac{2}{2\sigma'+\frac{\delta L_W^2}{p-1}\kappa_2(\mu)}\left(U(x)+\frac{L_W}{2}|x-\kappa_1(\mu)|^2\right)\right)dx}{\int_\mathbb{R} \exp\left(-\frac{2}{2\sigma'+\frac{\delta L_W^2}{p-1}\kappa_2(\mu)}\left(U(x)+\frac{L_W}{2}|x-\kappa_1(\mu)|^2\right)\right)dx}\\
    \kappa_2(\mu)=&\frac{\int_\mathbb{R}\left(x-\kappa_1(\mu)\right)^2\exp\left(-\frac{2}{2\sigma'+\frac{\delta L_W^2}{p-1}\kappa_2(\mu)}\left(U(x)+\frac{L_W}{2}|x-\kappa_1(\mu)|^2\right)\right)dx}{\int_\mathbb{R} \exp\left(-\frac{2}{2\sigma'+\frac{\delta L_W^2}{p-1}\kappa_2(\mu)}\left(U(x)+\frac{L_W}{2}|x-\kappa_1(\mu)|^2\right)\right)dx}.
\end{align*}
Thus, given two parameters $\delta$ and $p$, if $\mu$ is a stationary distribution for \eqref{eq:NL} with diffusion coefficient $\sigma$, it is also a stationary distribution for \eqref{eq:Eff} with diffusion coefficient $\sigma'=\sigma-\frac{\delta L_W^2}{2(p-1)}\kappa_2(\mu)$.
\end{proof}

Unfortunately, this lemma does not directly imply the existence of a phase transition for \eqref{eq:Eff}. Several issues arise :
\begin{itemize}
    \item the existence of a symmetric stationary distribution $\mu_{\sigma,0}$ for \eqref{eq:NL} with diffusion coefficient $\sigma>0$ only yields the existence of a symmetric stationary distribution for \eqref{eq:Eff} for a specific diffusion coefficient $\sigma'=\sigma-\frac{\delta L_W^2}{2(p-1)}\kappa_2(\mu)$. We  need to show that, for any $\sigma>0$, there exists a symmetric stationary distribution for \eqref{eq:Eff}.
    \item likewise, the existence of non symmetric stationary distribution for \eqref{eq:Eff} is only ensured for specific diffusion coefficients.
    \item we cannot infer the uniqueness of the symmetric stationary distribution for \eqref{eq:Eff} from the uniqueness of the symmetric  stationary distribution for \eqref{eq:NL}. Given $\sigma'>0$, there may \textit{a priori} be two stationary distributions for \eqref{eq:Eff}, denoted $\mu_1$ and $\mu_2$, with different variance $\kappa_2(\mu_1)\neq \kappa_2(\mu_2)$, which thus correspond to two different symmetric stationary distributions for \eqref{eq:NL} with diffusion coefficients ${\sigma_1=\sigma'+\frac{\delta L_W^2}{2(p-1)}\kappa_2(\mu_1)\neq \sigma'+\frac{\delta L_W^2}{2(p-1)}\kappa_2(\mu_2)=\sigma_2}$.
\end{itemize}

We therefore dedicate the remainder of this document to the proof of Theorem~\ref{thm:phase_trans_eff}.

%
%
%Subsection
%
%

\subsection{Some results on the stationary distribution(s) of \eqref{eq:NL}}

The study of the critical parameter of \eqref{eq:Eff} relies, by Lemma~\ref{lem:equiv_stat_NL_Eff}, on the study of the one of \eqref{eq:NL}. We gather here some results concerning the latter. They are numerically illustrated in Figure~\ref{Fig:analyse_tugaut}, and the proofs are postponed to Appendix~\ref{app:proofs_NL}.

\begin{figure}
\centering
\includegraphics[width=\linewidth]{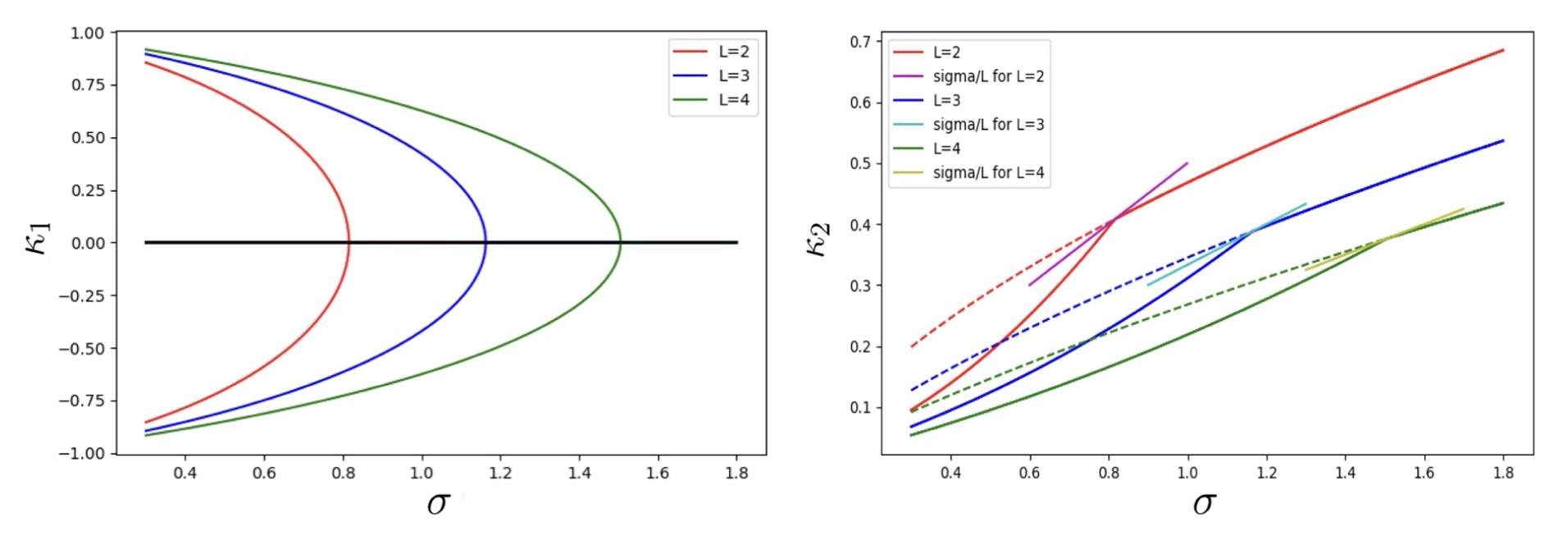}
\caption{\textbf{Left} : The means of $\mu_{\sigma,\pm}$ as a function of $\sigma$ for different values of $L_W$, as given in Theorem~\ref{thm:Tugaut}. \textbf{Right} : The variances of $\mu_{\sigma,\pm}$ as a function of $\sigma$ for different values of $L_W$. In dotted line the variance of $\mu_{\sigma,0}$, and in solid line the variance of $\mu_{\sigma,\pm}$.}
\label{Fig:analyse_tugaut}
\end{figure}

%Lemma

\begin{lemma}\label{lem:res_NL}
We have the following results concerning the stationary distribution(s) of \eqref{eq:NL}. 
\begin{itemize}
    \item \textbf{Symmetry. } We have $\kappa_1(\mu_{\sigma,+})=-\kappa_1(\mu_{\sigma,-})$ and $\kappa_2(\mu_{\sigma,+})=\kappa_2(\mu_{\sigma,-})$.
    \item \textbf{Moment bound.} Let the symbol $*\in\{0,\pm\}$. Consider $\mu_{\sigma,*}$ given in Theorem~\ref{thm:Tugaut}, and $\kappa_1(\mu_{\sigma,*})$ (resp. $\kappa_2(\mu_{\sigma,*})$) the corresponding mean (resp. variance). There exists $C_{\kappa_1}, C_{\kappa_2}>0$ such that for $\sigma\in[0,\sigma_c]$ we have 
    \begin{equation}\label{eq:borne_m1_tug}
        \left|\kappa_1(\mu_{\sigma,*})\right|\leq C_{\kappa_1},\ \ \ \text{ and }\ \ \ \left|\kappa_2(\mu_{\sigma,*})\right|\leq C_{\kappa_2}.
    \end{equation}
    \item \textbf{Critical variance.} We have the equality 
    \begin{equation}\label{eq:val_var_crit}
        \kappa_2\left(\mu_{\sigma_c}\right)=\frac{\sigma_c}{L_W}.
    \end{equation}
    Furthermore,  for $\sigma<\sigma_c$ we have
    $\kappa_2\left(\mu_{\sigma,\pm}\right)<\frac{\sigma}{L_W}$ and $\kappa_2\left(\mu_{\sigma,0}\right)>\frac{\sigma}{L_W}$, and for $\sigma>\sigma_c$ we have $\kappa_2\left(\mu_{\sigma,0}\right)<\frac{\sigma}{L_W}$.
    \item \textbf{Continuity.} The function $\sigma\mapsto \kappa_1(\mu_{\sigma,+})$, with the convention $\mu_{\sigma,+}=\mu_{\sigma,0}$ for $\sigma\geq\sigma_c$, is continuous on $]0,\infty[$. In particular, this also yields the continuity of $\sigma\mapsto \kappa_2(\mu_{\sigma,+})=f_2(\sigma,\kappa_1(\mu_{\sigma,+}))$.
    \item \textbf{Lipschitz continuity.} Let $\sigma_0>0$. The functions $\sigma\mapsto \kappa_2(\mu_{\sigma,0})$ and $\sigma\mapsto \kappa_2(\mu_{\sigma,\pm})$ are Lipschitz continuous, respectively on $[\sigma_0,\infty[$ and on $[\sigma_0,\sigma_c]$. More precisely, there exists $C>0$ such that for respectively $\sigma>\sigma_0$ and $\sigma\in]\sigma_0,\sigma_c[$ we have $\left|\frac{d}{d\sigma }\kappa_2(\mu_{\sigma,*})\right|\leq C$.
\end{itemize}
\end{lemma}

%Remark

\begin{remark}
The bound~\eqref{eq:borne_m1_tug}, combined with the knowledge of the fact that for $\sigma\geq\sigma_c$ there only exists a symmetric stationary distribution for \eqref{eq:NL} as well as Lemma~\ref{lem:equiv_stat_NL_Eff}, shows that we can restrict our study of the stationary distribution for both \eqref{eq:NL} and \eqref{eq:Eff} to a compact set of means $\kappa_1\in[-C_{\kappa_1},C_{\kappa_1}]$.
\end{remark}

%Remark

\begin{remark}
The main technical difficulty lies in the proof of the Lipschitz continuity of $\sigma\mapsto \kappa_2(\mu_{\sigma,0})$ and $\sigma\mapsto \kappa_2(\mu_{\sigma,\pm})$, since it turns out that the mean $\sigma\mapsto \kappa_1(\mu_{\sigma,\pm})$ is not Lipschitz continuous near the critical parameter $\sigma_c$ (See Figure~\ref{Fig:analyse_tugaut}). It therefore requires a careful estimation of the mean and variance around $\sigma_c$, and the proof is a section of its own (See Appendix~\ref{app:proof_lip}).
\end{remark}

%
%
%Subsection
%
%

\subsection{Phase transition for the effective dynamics}\label{sec:phase_trans_eff}

Let $\sigma_0>0$ and define, for $*\in\{0,\pm\}$, the function $g_{eff,*}:\sigma \mapsto \sigma-\frac{\delta L_W^2}{2(p-1)}\kappa_2(\mu_{\sigma,*})$. 

From Lemma~\ref{lem:equiv_stat_NL_Eff}, if $\mu_{\sigma,*}$ is a stationary distribution for \eqref{eq:NL}, then it is a stationary distribution for \eqref{eq:Eff} with diffusion coefficient $\sigma'=g_{eff,*}(\sigma)$. 

By Lemma~\ref{lem:res_NL}, $\sigma\mapsto \kappa_2(\mu_{\sigma,0})$ is a Lipschitz continuous function on $[\sigma_0,\infty[$ and $\sigma\mapsto \kappa_2(\mu_{\sigma,\pm})$ is a Lipschitz continuous function on $[\sigma_0,\sigma_c]$, and, more precisely, in both cases we obtain that $\left|\frac{d}{d\sigma}\kappa_2(\mu_{\sigma,*})\right|$ is bounded by some constant $C>0$.

In this case, the function $\sigma\mapsto g_{eff,*}(\sigma)$ is such that $g_{eff,*}'(\sigma)=1-\frac{\delta L_W^2}{2(p-1)}\frac{d}{d\sigma}\kappa_2(\mu_{\sigma,*})$ and thus $g_{eff,*}'(\sigma)\in\left[1-\frac{\delta L_W^2}{2(p-1)}C,1+\frac{\delta L_W^2}{2(p-1)}C\right]$. In particular, for $\frac{\delta}{p-1}$ sufficiently small, $g_{eff,*}(\sigma)$ is both an increasing continuous function and positive.

Thus, $g_{eff,0}$ and $g_{eff,\pm}$ are two injective functions. In particular, $g_{eff,\pm}$ is a bijection from $[\sigma_0,\sigma_c]$ to $[g_{eff,\pm}(\sigma_0),g_{eff,\pm}(\sigma_c)]$. 

Finally, notice that $g_{eff,\pm}(\sigma_c)=g_{eff,0}(\sigma_c)=\sigma_c\left(1-\frac{\delta L_W}{2(p-1)}\right)$, that $g_{eff,*}(\sigma_0)\leq \sigma_0$, and that, up to the additional assumption $\frac{2(\sigma_c-\sigma_0)}{\sigma_c L_W}>\frac{\delta}{p-1}$, we may assume $g_{eff,*}(\sigma_c)>\sigma_0$.\\

We may now state the following facts concerning the stationary distribution(s) for \eqref{eq:Eff}.

\begin{itemize}
    \item \textbf{There exists at least one symmetric stationary distribution.} Since $g_{eff,0}'\geq 1-\frac{\delta L_W^2}{2(p-1)}C$, $g_{eff,0}$ is an increasing function such that $g_{eff,0}(x)\xrightarrow[x\rightarrow\infty]{}\infty$. Thus, $g_{eff,0}$ is a bijection from $[\sigma_0,\infty[$ to $[g_{eff,0}(\sigma_0),\infty[$. Therefore, for all $\sigma\in[g_{eff,0}(\sigma_0),\infty[$, there exists $\tilde{\sigma}$ such that $g_{eff,0}(\tilde{\sigma})=\sigma$. In other words, $\mu_{\tilde{\sigma},0}$ is also a symmetric stationary distribution for \eqref{eq:Eff} with diffusion coefficient $\sigma$.
    
    \item \textbf{There exists at most one symmetric stationary distribution.} Let $\sigma\geq\sigma_0$ and $\kappa_1=0$, and assume there are two symmetric stationary distributions of \eqref{eq:Eff} with diffusion coefficient $\sigma$. This yields two  coefficients $\sigma',\sigma''\geq\sigma>0$ such that
\begin{align*}
    \sigma'=&\sigma+\frac{\delta L_W^2}{2(p-1)}\kappa_2\left(\mu^{\delta,p}_{\sigma,1}\right)\\
    \sigma''=&\sigma+\frac{\delta L_W^2}{2(p-1)}\kappa_2\left(\mu^{\delta,p}_{\sigma,2}\right),
\end{align*}
where $\mu^{\delta,p}_{\sigma,1}$ and $\mu^{\delta,p}_{\sigma,2}$ denote the two stationary distributions. We consider $\mu_{\sigma'}$ ($=\mu^{\delta,p}_{\sigma,1}$) and $\mu_{\sigma''}$ ($=\mu^{\delta,p}_{\sigma,2}$)  the corresponding (unique) symmetric stationary distributions of \eqref{eq:NL}. Because $\sigma'$ and $\sigma''$ are greater than $\sigma_0$, there exists a constant $K$, possibly depending on $\sigma_0$ and $L_W$, such that by Lemma~\ref{lem:res_NL}
\begin{align}
    \left|\kappa_2\left(\mu_{\sigma''}\right)-\kappa_2\left(\mu_{\sigma'}\right)\right|=&\left|f_2(\sigma'',0)-f_2(\sigma',0)\right|\leq K|\sigma''-\sigma'|\nonumber\\
    \text{ i.e. }\ \ \ \left|\kappa_2\left(\mu_{\sigma''}\right)-\kappa_2\left(\mu_{\sigma'}\right)\right|\leq&\frac{K\delta L_W^2}{2(p-1)}\left|\kappa_2\left(\mu^{\delta,p}_{\sigma,2}\right)-\kappa_2\left(\mu^{\delta,p}_{\sigma,1}\right)\right|.\label{eq:contr_unicite_dist}
\end{align}
Because the stationary distributions of \eqref{eq:Eff} are uniquely defined by their mean and variance, 
\begin{align*}
    \kappa_2\left(\mu_{\sigma''}\right)=\kappa_2\left(\mu^{\delta,p}_{\sigma,2}\right)\ \ \ \text{ and }\ \ \ \kappa_2\left(\mu_{\sigma'}\right)=\kappa_2\left(\mu^{\delta,p}_{\sigma,1}\right),
\end{align*}
and we obtain from \eqref{eq:contr_unicite_dist} that, for $\frac{\delta}{p-1}$ sufficiently small, $\kappa_2\left(\mu^{\delta,p}_{\sigma,1}\right)=\kappa_2\left(\mu^{\delta,p}_{\sigma,2}\right)$ and thus that $\mu^{\delta,p}_{\sigma,1}=\mu^{\delta,p}_{\sigma,2}$.

    \item \textbf{For $\sigma\in[g_{eff,\pm}(\sigma_0),g_{eff,\pm}(\sigma_c)[$, there exists at least two stationary distributions with nonzero mean.} Because $g_{eff,\pm}$ is a bijection, consider $g_{eff,\pm}^{-1}(\sigma)\in[\sigma_0,\sigma_c[$. There are three stationary distribution for \eqref{eq:NL} with diffusion coefficient $g_{eff,\pm}^{-1}(\sigma)$. By Lemma~\ref{lem:equiv_stat_NL_Eff}, $\mu_{g_{eff,\pm}^{-1}(\sigma),+}$ and $\mu_{g_{eff,\pm}^{-1}(\sigma),-}$ are also stationary distributions for \eqref{eq:Eff} with diffusion coefficient $\sigma$, and they have nonzero mean.

    \item \textbf{For $\sigma\in[g_{eff,\pm}(\sigma_0),g_{eff,\pm}(\sigma_c)[$, there exists at most two stationary distributions with nonzero mean.} By symmetry, it is sufficient to prove that there is at most one stationary distribution with positive mean. Assume there are two such solutions $\mu^{\delta,p}_{\sigma,+,1}$ and $\mu^{\delta,p}_{\sigma,+,2}$.
    
    Let, for $i\in\{1,2\}$,  $\sigma_i=\sigma+\frac{\delta L_W^2}{2(p-1)}\kappa_2(\mu^{\delta,p}_{\sigma,+,i})$. Then, $\mu^{\delta,p}_{\sigma,+,i}$ is a stationary distribution with a positive mean for \eqref{eq:NL} with diffusion coefficient $\sigma_i$, i.e $\mu^{\delta,p}_{\sigma,+,i}=\mu_{\sigma_i,+}$.
    
    Thus $\sigma=\sigma_i-\frac{\delta L_W^2}{2(p-1)}\kappa_2(\mu^{\delta,p}_{\sigma,+,i})=\sigma_i-\frac{\delta L_W^2}{2(p-1)}\kappa_2(\mu_{\sigma_i,+})=g_{eff,+}(\sigma_i)$. Since $g_{eff,+}$ is an injective function, we obtain that $\sigma_1=\sigma_2$. In particular, $\mu^{\delta,p}_{\sigma,+,1}$ and $\mu^{\delta,p}_{\sigma,+,2}$ are two stationary distribution with a positive mean for \eqref{eq:NL} with diffusion coefficient $\sigma_1=\sigma_2$, thus by uniqueness $\mu^{\delta,p}_{\sigma,+,1}=\mu^{\delta,p}_{\sigma,+,2}$.
    
    \item \textbf{For $\sigma\geq g_{eff,+}(\sigma_c)$, there does not exists stationary distribution with nonzero mean}. The result is direct if $\sigma\geq\sigma_c$, because if $\mu$ is a stationary measure for \eqref{eq:Eff} with diffusion coefficient $\sigma$, it is is a stationary measure for \eqref{eq:NL} with diffusion coefficient $\sigma+\frac{\delta L_W^2}{2(p-1)}\kappa_2(\mu)\geq\sigma>\sigma_c$, hence it cannot have a nonzero mean.
    
    Assume $\sigma_c>\sigma\geq g_{eff,+}(\sigma_c)>\sigma_0$ and that there exists such a solution $\mu^{\delta,p}_{\sigma,+}$. Consider 
    \begin{equation}\label{eq:surcrit_1}
        \sigma'=\sigma+\frac{\delta L_W^2}{2(p-1)}\kappa_2(\mu^{\delta,p}_{\sigma,+}),
    \end{equation}
    such that $\mu^{\delta,p}_{\sigma,+}=\mu_{\sigma',+}$ is a stationary distribution with positive mean for \eqref{eq:NL}. We thus necessarily have $\sigma'<\sigma_c$. Let 
    \begin{equation}\label{eq:surcrit_2}
        \tilde{\sigma}=g_{eff,+}(\sigma')=\sigma'-\frac{\delta L_W^2}{2(p-1)}\kappa_2(\mu_{\sigma',+}).
    \end{equation}    
    We obtain $\mu^{\delta,p}_{\sigma,+}=\mu_{\sigma',+}=\mu^{\delta,p}_{\tilde{\sigma},+}$. Since $g_{eff,+}$ is increasing we have $g_{eff,+}(\sigma_c)>g_{eff,+}(\sigma')=\tilde{\sigma}$, we obtain from \eqref{eq:surcrit_2}
    \begin{align*}
        g_{eff,+}(\sigma_c)>&\sigma'-\frac{\delta L_W^2}{2(p-1)}\kappa_2(\mu_{\sigma',+})=\sigma'-\frac{\delta L_W^2}{2(p-1)}\kappa_2(\mu^{\delta,p}_{\sigma,+}),\\
        \text{ i.e }\ \ \ \kappa_2(\mu^{\delta,p}_{\sigma,+})>&\frac{2(p-1)}{\delta L_W^2}\left(\sigma'-g_{eff,+}(\sigma_c)\right).
    \end{align*}
    Plugging that back into \eqref{eq:surcrit_1}, we obtain
    \begin{align*}
        \sigma'>\sigma+\sigma'-g_{eff,+}(\sigma_c),\ \ \ \text{ i.e }\ \ \ \sigma<g_{eff,+}(\sigma_c),
    \end{align*}
    which contradicts the initial assumption.
\end{itemize}
This concludes the proof of Theorem~\ref{thm:phase_trans_eff}.

%
%Appendix
%

\newpage

\appendix

%
%Section
%

\section{Technical lemmas}\label{app:technical}

We start with this slight extension of the central limit theorem: 

%Lemma

\begin{lemma}\label{lem:cv_mom}
Let $(X_i)_{i\geq 1}$ be a sequence of i.i.d random variables in $\mathbb{R}$ such that $\mathbb{E}X_1=0$ and $\mathbb{E}(|X_1|^2)=1$. Assume also that $\mathbb{E}(|X_1|^4)<\infty$. Denote $Z_p=\frac{1}{\sqrt{p}}\sum_{i=1}^{p}X_i$.

Then, we have for $Z\sim\mathcal{N}(0,1)$
\begin{align*}
\mathbb{E}\left(|Z_p|^2\right)=\mathbb{E}\left(|Z|^2\right),\ \ \ \text{ and }\ \ \ \mathbb{E}\left(|Z_p|^4\right)=\mathbb{E}\left(|Z|^4\right)+O\left(\frac{1}{p}\right).
\end{align*}
This in particular also yields $\mathbb{E}\left(|Z_p|\right)\xrightarrow[p\rightarrow\infty]{}\mathbb{E}\left(|Z|\right)$ and $\mathbb{E}\left(|Z_p|^3\right)\xrightarrow[p\rightarrow\infty]{}\mathbb{E}\left(|Z|^3\right)$
\end{lemma}

%Proof

\begin{proof}
Direct computations give $\mathbb{E}\left(|Z_p|^2\right)=\mathbb{E}\left(|Z|^2\right)$. 
Likewise, we may explicitly compute $\mathbb{E}|Z_p|^4$. Keeping only the terms with nonzero expectation, we have
\begin{align*}
    \mathbb{E}|Z_p|^4=&\frac{1}{p^2}\sum_{i=1}^p\mathbb{E}|X_i|^4+\frac{6}{p^2}\sum_{i>j}\mathbb{E}|X_i|^2\mathbb{E}|X_j|^2\\
    =&\frac{\mathbb{E}|X_1|^4}{p}+\frac{6}{p^2}\frac{p(p-1)}{2}.
\end{align*}
Noticing that $\mathbb{E}|Z|^4=3$ yields the convergence $\mathbb{E}\left(|Z_p|^4\right)=\mathbb{E}\left(|Z|^4\right)+O\left(\frac{1}{p}\right)$.
Thus, we have both 
\begin{itemize}
    \item $Z_p$ converges in law to $Z\sim\mathcal{N}(0,1)$,
    \item and the convergence of the fourth moment $\mathbb{E}\left|Z_{p}\right|^4\xrightarrow[p\rightarrow\infty]{}\mathbb{E}\left|Z\right|^4$.
\end{itemize}
By \cite[Theorem~6.9]{villani2008optimal}, we have the convergence in $L^4$ Wasserstein distance of the law of $Z_p$ to a law $\mathcal{N}(0,1)$. This implies the convergence in both $L^1$ Wasserstein distance and $L^3$ Wasserstein distance, thus the convergence of the first and third moments of $Z_p$.
\end{proof}

%Lemma

\begin{lemma}\label{lem:der_3}
The function $\sigma\in]0,\infty[\mapsto\partial_{\kappa}^3f_1(\sigma,0)$ (with $f_1$ given by \eqref{eq:def_f1}) is continuous and satisfies
\begin{equation*}
    \forall\sigma>0,\ \ \ \partial_{\kappa}^3f_1(\sigma,0)<0.
\end{equation*}
\end{lemma}

%Proof

\begin{proof}
We have
\begin{align*}
    \partial_{\kappa}^3f_1(\sigma,0)=&\left(\frac{L_W}{\sigma}\right)^3\left(\frac{\int_\mathbb{R}x^4 g}{\int_\mathbb{R} g}-3\left(\frac{\int_\mathbb{R}x^2 g}{\int_\mathbb{R} g}\right)^2\right)\,,
\end{align*}
with $g$ given by \eqref{eq:def_g}. We wish to prove
\begin{align*}
    A(\sigma,L_W):=\frac{\int_\mathbb{R}x^4 g\int_\mathbb{R} g}{\left(\int_\mathbb{R}x^2 g\right)^2}<3.
\end{align*}
Remark that $A(\sigma,L_W)$ is by definition the kurtosis of a random variable with probability density $\frac{g}{\int g}$. 

Let us rewrite, for $\alpha>0$
\begin{align*}
    U(x)+\frac{L_W}{2}x^2=&\frac{x^4}{4}+\frac{L_W-1}{2}x^2\\
    =&\frac{x^4}{4}+\frac{L_W-1-\alpha}{2}x^2+\left(\frac{L_W-1-\alpha}{2}\right)^2+\frac{\alpha}{2}x^2-\left(\frac{L_W-1-\alpha}{2}\right)^2\\
    =&\frac{1}{4}\left(x^2+L_W-1-\alpha\right)^2+\frac{\alpha}{2}x^2-\left(\frac{L_W-1-\alpha}{2}\right)^2,
\end{align*}
such that
\begin{align*}
    g(x,\sigma,0)=\exp\left(-\frac{1}{4\sigma}\left(x^2+L_W-1-\alpha\right)^2\right)\exp\left(\plb{\frac{1}{\sigma}}\left(\frac{L_W-1-\alpha}{2}\right)^2\right)e^{-\frac{\alpha}{2\sigma}x^2}.
\end{align*}
This way we can write
\begin{align*}
    A(\sigma,L_W)=&\frac{\mathbb{E}\left(\exp\left(-\frac{1}{4\sigma}\left(Y^2+L_W-1-\alpha\right)^2\right)\right)\mathbb{E}\left(Y^4\exp\left(-\frac{1}{4\sigma}\left(Y^2+L_W-1-\alpha\right)^2\right)\right)}{\mathbb{E}\left(Y^2\exp\left(-\frac{1}{4\sigma}\left(Y^2+L_W-1-\alpha\right)^2\right)\right)^2},\\
    &\hspace{8cm}\text{ with } \ Y\sim\mathcal{N}\left(0,\frac{\sigma}{\alpha}\right),\\
    =&\frac{\mathbb{E}\left(\exp\left(-\frac{1}{4\sigma}\left(\frac{\sigma}{\alpha}X^2+L_W-1-\alpha\right)^2\right)\right)\mathbb{E}\left(X^4\exp\left(-\frac{1}{4\sigma}\left(\frac{\sigma}{\alpha}X^2+L_W-1-\alpha\right)^2\right)\right)}{\mathbb{E}\left(X^2\exp\left(-\frac{1}{4\sigma}\left(\frac{\sigma}{\alpha}X^2+L_W-1-\alpha\right)^2\right)\right)^2},\\
    &\hspace{8cm}\text{ with } \ X\sim\mathcal{N}\left(0,1\right).
\end{align*}
We have
\begin{align*}
    -\frac{1}{4\sigma}\left(\frac{\sigma}{\alpha}X^2+L_W-1-\alpha\right)^2=-\frac{1}{4\sigma}\left(\frac{\sigma}{\alpha}\right)^2\left(X^2+\frac{\alpha\left(L_W-1-\alpha\right)}{\sigma}\right)^2.
\end{align*}
We thus choose $\alpha=\frac{1}{2}\left(L_W-1+\sqrt{(L_W-1)^2+4\plb{\sigma}}\right)>0$ in order to ensure $\frac{\alpha\left(L_W-1-\alpha\right)}{\sigma}=-1$. Finally, denoting \[\beta(\sigma,L_W)=\plb{\left(\frac{L_W-1}{\sqrt{\sigma}}+\sqrt{\left(\frac{L_W-1}{\sqrt{\sigma}}\right)^2+4}\right)^{-2}}>0\,,\]
we have
\begin{align}
    A(\sigma,L_W)=&\frac{\mathbb{E}\left(\exp\left(-\beta(\sigma,L_W)\left(X^2-1\right)^2\right)\right)\mathbb{E}\left(X^4\exp\left(-\beta(\sigma,L_W)\left(X^2-1\right)^2\right)\right)}{\mathbb{E}\left(X^2\exp\left(-\beta(\sigma,L_W)\left(X^2-1\right)^2\right)\right)^2},\ \ \ X\sim\mathcal{N}\left(0,1\right)\nonumber\\
    :=&A(\beta(\sigma,L_W))\label{eq:A_beta}.
\end{align}
The quantity $A$ given above can be expressed as a function of $\beta(\sigma,L_W)\in]0,\infty[$, that we denote $A(\beta)$. We may then numerically check that $A(\beta)<3$ for all $\beta>0$. (see Figure~\ref{Fig:A_beta}). 

Notice that the function $\beta(\sigma,L_W)$, which is in reality a function of the quantity $\frac{L_W-1}{\sqrt{\sigma}}$, is a bijection from $\frac{L_W-1}{\sqrt{\sigma}}\in\mathbb{R}$ to $]0,\infty[$, which satisfies $\beta(\sigma,L_W)\xrightarrow[\frac{L_W-1}{\sqrt{\sigma}}\rightarrow\infty]{}0$. And, for $\beta(\sigma,L_W)=0$, direct calculations knowing the moments of the Gaussian law yield $A(0)=3$ and $A'(0)=\plb{-24}<0$.
\end{proof}

\begin{figure}
\centering
\includegraphics[width=0.7\linewidth]{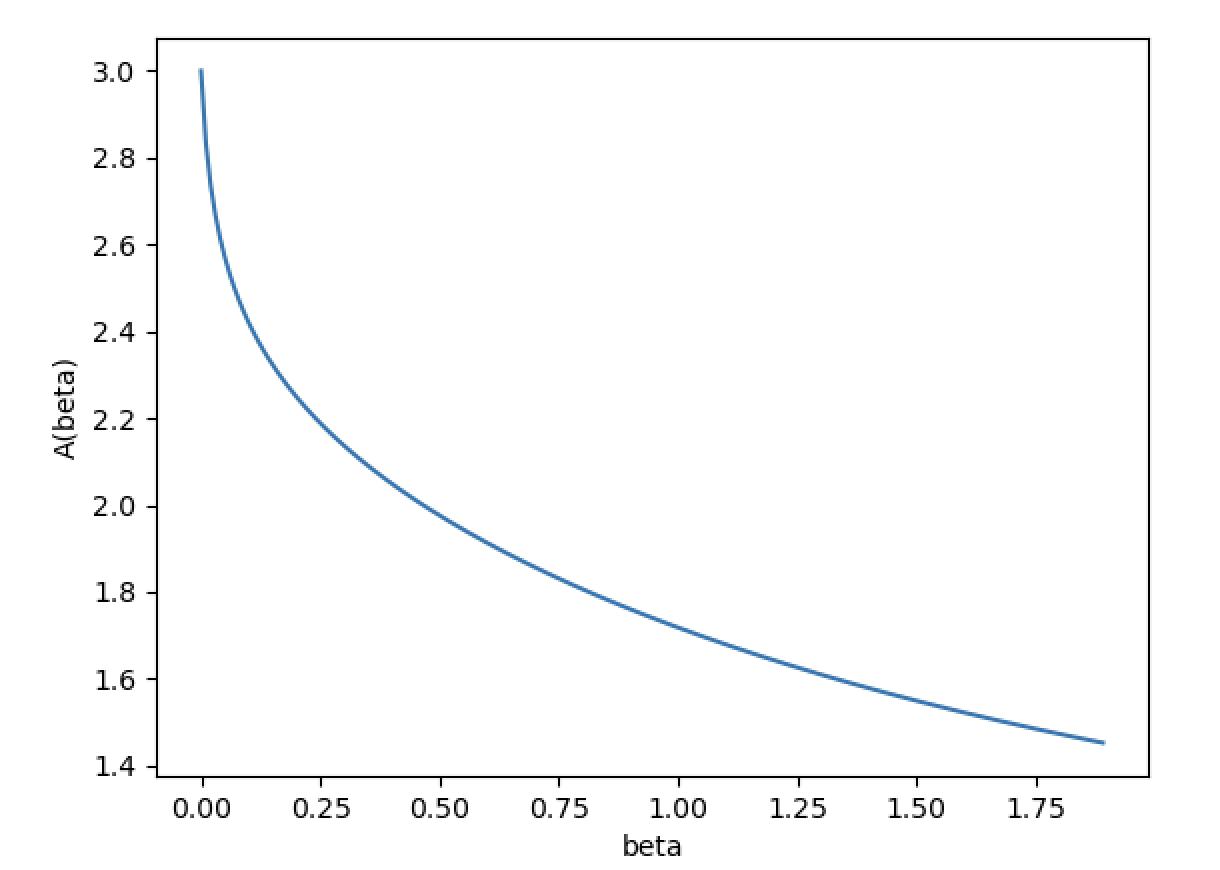}
\caption{Numerical simulation of the quantity given in \eqref{eq:A_beta}.}
\label{Fig:A_beta}
\end{figure}

%
%Lemma
%

\begin{lemma}\label{lem:d_sigma_d_m_xi}
Consider the function $F_1:\sigma\mapsto \partial_\kappa f_1(\sigma,0)$. It is continuously differentiable and, for $\sigma>0$, satisfies  $F_1'(\sigma)<0$.
\end{lemma}

%
%Lemma
%

\begin{proof}
We have
\begin{align*}
    F_1(\sigma)=\frac{L_W}{\sigma}\frac{\int x^2\exp\left(-\frac{x^4}{4\sigma}+\frac{1-L_W}{2\sigma}x^2\right)dx}{\int \exp\left(-\frac{x^4}{4\sigma}+\frac{1-L_W}{2\sigma}x^2\right)dx}.
\end{align*}
Consider the change of variable $y=\frac{x}{\sqrt{\sigma}}$. We have
\begin{align*}
    F_1(\sigma)=&\frac{L_W}{\sigma}\frac{\int \sigma y^2\exp\left(-\sigma^2\frac{y^4}{4\sigma}+\frac{1-L_W}{2\sigma}\sigma y^2\right)dy}{\int \exp\left(-\sigma^2\frac{y^4}{4\sigma}+\frac{1-L_W}{2\sigma}\sigma y^2\right)dy}\\
    =&L_W\frac{\int y^2\exp\left(-\frac{\sigma}{4}y^4+\frac{1-L_W}{2}y^2\right)dy}{\int \exp\left(-\frac{\sigma}{4}y^4+\frac{1-L_W}{2}y^2\right)dy},
\end{align*}
which then yields
\begin{align*}
    F_1'(\sigma)=&\frac{L_W}{4}\left(-\mathbb{E}\left(Y^6\right)+\mathbb{E}\left(Y^2\right)\mathbb{E}\left(Y^4\right)\right),
\end{align*}
where $Y$ is a random variable with probability density (up to renormalization) $\exp\left(-\frac{\sigma}{4}y^4+\frac{1-L_W}{2}y^2\right)dy$. By Jensen inequality (in a strictly convex case with a non almost surely constant random variable), we have $F_1'(\sigma)<0$.
\end{proof}

%
%
%Section
%
%

\section{Proofs of Lemma~\ref{lem:res_NL} }\label{app:proofs_NL} 

In this section we prove the various results of Lemma~\ref{lem:res_NL}.

%
%
%Subsection
%
%

\subsection{Moment bounds, critical variance and continuity}

To prove the first properties stated in Lemma~\ref{lem:res_NL}, we recall the following result, extracted from Theorem~2.1 of \cite{tugaut} and its proof.

%Lemma

\begin{lemma}\label{lem:tech_tug}
The equation (with unknown $\sigma$)
\begin{equation}\label{eq:cdt_crit_tug}
    \int_{\mathbb{R}^+}\left(x^2-\frac{1}{2L_W}\right)\exp\left((1-L_W)x^2-\sigma x^4\right)dx=0,
\end{equation}
admits a unique solution, that is the critical value $\sigma_c$.

Finally, consider the function 
\begin{equation}\label{eq:def_xi}
    \xi(\sigma,\kappa)=\int_{\mathbb{R}}(x-\kappa)\exp\left(-\frac{1}{\sigma}\left(U(x)+\frac{L_W}{2}x^2-L_Wx\kappa\right)\right)dx.
\end{equation}
We have the following properties on $\xi$ :
\begin{itemize}
    \item The function $\sigma\mapsto\partial_\kappa\xi(\sigma,0)$ is decreasing : for $\sigma<\sigma_c$ we have $\partial_\kappa\xi(\sigma,0)>0$, for $\sigma>\sigma_c$ we have $\partial_\kappa\xi(\sigma,0)<0$, and finally $\partial_\kappa\xi(\sigma_c,0)=0$. 
    \item For $\sigma\geq\sigma_c$ and $\kappa\geq0$, the function $\kappa\mapsto\xi(\sigma,\kappa)$ is decreasing (which, in fact, ensures uniqueness of the stationary solution for \eqref{eq:NL}).
    \item For $\sigma<\sigma_c$ and $\kappa\geq0$, the function $\kappa\mapsto\xi(\sigma,\kappa)$ is increasing and then decreasing (which,in fact, ensures the thirdness of the stationary solution for \eqref{eq:NL}).
\end{itemize}
\end{lemma}

\paragraph{Moment bound :} Consider $(\bar{X}_t)_t$ the solution of \eqref{eq:NL}. Itô's formula yields
\begin{align*}
    dU(\bar{X}_t)=A_tdt+dM_t,
\end{align*}
where $M_t$ is a continuous local martingale and
\begin{align*}
    A_t=-U'\left(\bar{X}_t\right)^2-L_W\left(\bar{X}_t-\mathbb{E}\left(\bar{X}_t\right)\right)U'\left(\bar{X}_t\right)+\sigma U''\left(\bar{X}_t\right).
\end{align*}
There exists $\lambda>0$ and $C>0$, both independent of $\sigma\in[0,\sigma_c]$, such that for all $x\in\mathbb{R}$
\begin{align*}
    \sigma U''(x)+2\lambda U(x)\leq \frac{U'(x)^2}{2}+C,\hspace{1cm}
    \text{ and }\hspace{1cm} 2L_W^2x^2\leq \lambda U(x)+C.
\end{align*}
Consider for instance $\lambda=1$ and $C=\max\left(2\sigma+(2\sigma)^{3/2}-\frac{1}{2},\frac{(1+4 L_W^2)^2}{4}\right)$ for $U(x)=\frac{x^4}{4}-\frac{x^2}{2}$. Thus 
\begin{align*}
    A_t\leq C-\lambda U\left(\bar{X}_t\right)+\left(L_W^2\left(\bar{X}_t^2+\mathbb{E}\left(\bar{X}_t\right)^2\right)-\lambda U\left(\bar{X}_t\right) \right),
\end{align*}
and, using Fatou's lemma to deal with the local martingale, finally we obtain thanks to Gronwall's lemma
\begin{align*}
    \mathbb{E}U\left(\bar{X}_t\right)\leq e^{-\lambda t}\mathbb{E}U\left(\bar{X}_0\right)+\frac{2C}{\lambda}.
\end{align*}
Since $\mathbb{E}\left(\bar{X}_t\right)^2\leq\mathbb{E}\left(\bar{X}_t^2\right)\leq \frac{1}{2 L_W^2}\left(\lambda \mathbb{E}U\left(\bar{X}_t\right)+C\right)$, and considering $\bar{X}_0$ distributed according to a stationary distribution, we may conclude.

\paragraph{Value of $\kappa_2\left(\mu_{\sigma_c}\right)$ :} We rewrite Equation~\eqref{eq:cdt_crit_tug} defining $\sigma_c$, first by using the symmetry in $x$ to obtain
\begin{align*}
    \int_{\mathbb{R}}\left(x^2-\frac{1}{2L_W}\right)\exp\left((1-L_W)x^2-\sigma x^4\right)dx=0,
\end{align*}
and then, by a change of variable $x=\frac{y}{\sqrt{2\sigma}}$, this is equivalent to
\begin{align*}
    \int_{\mathbb{R}}\left(y^2-\frac{\sigma}{L_W}\right)\exp\left(-\frac{1}{\sigma}\left(\frac{y^4}{4}-\frac{1-L_W}{2}y^2\right)\right)dy=0.
\end{align*}
Finally, this amounts to having
\begin{align*}
    \frac{\sigma}{L_W}=\frac{\int_{\mathbb{R}}y^2\exp\left(-\frac{1}{\sigma}\left(\frac{y^4}{4}-\frac{1-L_W}{2}y^2\right)\right)dy}{\int_{\mathbb{R}}\exp\left(-\frac{1}{\sigma}\left(\frac{y^4}{4}-\frac{1-L_W}{2}y^2\right)\right)dy},
\end{align*}
which, since $\kappa_1(\mu_{\sigma_c})=0$, yields the value of $\kappa_2(\mu_{\sigma_c})$.

Consider then the function $\xi$ given \eqref{eq:def_xi}. We have
$\xi(\sigma,\kappa)=(f_1(\sigma,\kappa)-\kappa)\int g(x,\sigma,\kappa)dx$, and thus
\begin{align*}
    \partial_\kappa \xi(\sigma,\kappa)=(\partial_\kappa f_1(\sigma,\kappa)-1)\int g(x,\sigma,\kappa)dx+(f_1(\sigma,\kappa)-\kappa)\int \partial_\kappa g(x,\sigma,\kappa)dx.
\end{align*}
Considering the equation above for $\kappa=\kappa_1(\mu_{\sigma,*})$, we obtain
\begin{align*}
    \partial_\kappa \xi(\sigma,\kappa_1(\mu_{\sigma,*}))=(\partial_\kappa f_1(\sigma,\kappa_1(\mu_{\sigma,*}))-1)\int g(x,\sigma,\kappa_1(\mu_{\sigma,*}))dx.
\end{align*}
We may compute the derivatives of $f_1$ (see \eqref{eq:RBM_der_part} below), and obtain
\begin{align*}
    \partial_\kappa  \xi(\sigma,\kappa_1(\mu_{\sigma,*}))=\left(\frac{L_W}{\sigma}\kappa_2(\mu_{\sigma,*})-1\right)\int g(x,\sigma,\kappa_1(\mu_{\sigma,*}))dx.
\end{align*}
The values of $\partial_\kappa  \xi(\sigma,\kappa)$ for $\kappa=0$ and $\kappa=\kappa_1(\mu_{\sigma,+})$ depending on $\sigma$, as given in Lemma~\ref{lem:tech_tug}, yields the result.

\paragraph{Continuity of the moments :}Notice that $f_1$ given in \eqref{eq:def_f1} is continuous on $(\sigma,\kappa)\in\mathbb{R}^{+,*}\times\mathbb{R}^{+}$. We start by proving the continuity of $\sigma\mapsto \kappa_1(\mu_{\sigma,+})$, with the convention $\mu_{\sigma,+}=\mu_{\sigma,0}$ for $\sigma>\sigma_c$. In this latter case, the function $\sigma\mapsto \kappa_1(\mu_{\sigma,+})$ is trivially continuous as $\kappa_1(\mu_{\sigma,+})=0$. 

Let us show the continuity at the point $\sigma_c$. Let $(\sigma_n)_{n\in\mathbb{N}}$ be a sequence of positive real numbers such that $\sigma_n\xrightarrow[n\rightarrow\infty]{}\sigma_c$, and consider the (bounded) sequence $\left(\kappa_1(\mu_{\sigma_n,+})\right)_n$. Up to extraction, we can assume $\kappa_1(\mu_{\sigma_n,+})\xrightarrow[n\rightarrow\infty]{}\kappa_1\geq0$. We have, by definition, $\kappa_1(\mu_{\sigma_n,+})=f_1(\sigma_n,\kappa_1(\mu_{\sigma_n,+}))$ and, by considering the limit $n\rightarrow\infty$, thanks to the continuity of $f_1$, we obtain $\kappa_1=f_1(\sigma_c,\kappa_1)$. Uniqueness of the fixed point for $\sigma_c$ then ensures $\kappa_1=0=\kappa_1(\mu_{\sigma_c,+})$. Hence we obtain the desired continuity.

We now consider $\sigma<\sigma_c$. Assume there exists $\epsilon>0$ and a sequence $(\sigma_n)_{n\in\mathbb{N}}$ such that $\sigma_n\xrightarrow[n\rightarrow\infty]{}\sigma$ and $|\kappa_1(\mu_{\sigma_n,+})-\kappa_1(\mu_{\sigma,+})|>\epsilon$. Again, up to extraction, we have $\kappa_1(\mu_{\sigma_n,+})\xrightarrow[n\rightarrow\infty]{}\kappa_1\geq0$ and, since $\kappa_1$ is a fixed point that cannot be $\kappa_1(\mu_{\sigma,+})$, we have $\kappa_1=0$. Consider the (at least) twice continuously differentiable function $\xi$ given in \eqref{eq:def_xi}. On one hand, we have by continuity $\partial_\kappa\xi(\sigma_n, \kappa_1(\mu_{\sigma_n,+}))\xrightarrow[n\rightarrow\infty]{}\partial_\kappa\xi(\sigma, 0)>0$. On the other hand, by the properties of $\xi$ given in Theorem~\ref{lem:tech_tug} , we have $\partial_\kappa\xi(\sigma_n, \kappa_1(\mu_{\sigma_n,+}))<0$. Hence a contradiction, and $\kappa_1(\mu_{\sigma_n,+})\xrightarrow[n\rightarrow\infty]{}\kappa_1(\mu_{\sigma,+})$ for any sequence $\sigma_n\xrightarrow[n\rightarrow\infty]{}\sigma$. We thus obtain the continuity.

%
%
%Subsubsection
%
%

\subsection{On the variance of the stationary distribution(s) of \eqref{eq:NL}}\label{app:proof_lip}

Let $\sigma_0>0$, and let us show that the functions $\sigma\mapsto \kappa_2(\mu_{\sigma,0})$ and $\sigma\mapsto \kappa_2(\mu_{\sigma,\pm})$ are Lipschitz continuous. This is useful, as can be seen in Section~\ref{sec:phase_trans_eff}, in proving that there exists a phase transition for the effective dynamics \eqref{eq:Eff}.

Throughout this section, the constant $C$ holds no importance and may change from one line to the next.

\paragraph{We start by showing that $\sigma\mapsto \kappa_2(\mu_{\sigma,0})$ is Lipschitz continuous.}

%Lemma

\begin{lemma}\label{lem:var_lip}
Let $\sigma_0>0$ and $\kappa_1\in[-C_{\kappa_1},C_{\kappa_1}]$ (where $C_{\kappa_1}$ is given in \eqref{eq:borne_m1_tug}). The function $\sigma\mapsto f_2(\sigma,\kappa_1)$ is Lipschitz continuous on $[\sigma_0,\infty[$ uniformly in $\kappa_1\in[-C_{\kappa_1},C_{\kappa_1}]$.
\end{lemma}

Applying this lemma for $\kappa_1=0$ yields the desired Lipschitz continuity for $\kappa_2(\mu_{\sigma,0})$.

%Proof

\begin{proof}[Proof of Lemma~\ref{lem:var_lip}]
Recall $g$ defined in \eqref{eq:def_g}. and consider $C=\frac{1+2L_W\kappa_1^2}{4}$, a constant such that, for $U$ given by \eqref{eq:def_U_W}, we ensure $U(x)+\frac{L_W}{2}x^2-L_Wx\kappa_1+C\geq0$. We have
\begin{align*}
    f_2(\sigma,\kappa_1)=\frac{\int_\mathbb{R}(x-\kappa_1)^2g(x,\sigma,\kappa_1)e^{-\frac{C}{\sigma}}dx}{\int_\mathbb{R} g(x,\sigma,\kappa_1)e^{-\frac{C}{\sigma}}dx},
\end{align*}
and thus
\begin{align}
    \left|\partial_\sigma f_2(\sigma,\kappa_1)\right|=&\frac{1}{\sigma^2}\left|\frac{\int_\mathbb{R}(x-\kappa_1)^2\left(U(x)+\frac{L_W}{2}x^2-L_Wx\kappa_1+C\right) g(x,\sigma,\kappa_1)dx\int_\mathbb{R} g(x,\sigma,\kappa_1)dx}{\left(\int_\mathbb{R} g(x,\sigma,\kappa_1)dx\right)^2}\right.\nonumber\\
    &\hspace{1cm}\left.-\frac{\int_\mathbb{R}(x-\kappa_1)^2 g(x,\sigma,\kappa_1)dx\int_\mathbb{R} \left(U(x)+\frac{L_W}{2}x^2-L_Wx\kappa_1+C\right)g(x,\sigma,\kappa_1)dx}{\left(\int_\mathbb{R} g(x,\sigma,\kappa_1)dx\right)^2}\right|\nonumber\\
    \leq&\frac{1}{\sigma^2}\frac{\int_\mathbb{R}(x-\kappa_1)^2\left(U(x)+\frac{L_W}{2}x^2-L_Wx\kappa_1+C\right) g(x,\sigma,\kappa_1)dx}{\int_\mathbb{R} g(x,\sigma,\kappa_1)dx}\label{eq:RBM_lip_estim_1}\\
    &+\frac{1}{\sigma^2}\frac{\int_\mathbb{R}(x-\kappa_1)^2 g(x,\sigma,\kappa_1)dx}{\int_\mathbb{R} g(x,\sigma,\kappa_1)dx}\frac{\int_\mathbb{R}\left(U(x)+\frac{L_W}{2}x^2-L_Wx\kappa_1+C\right) g(x,\sigma,\kappa_1)dx}{\int_\mathbb{R} g(x,\sigma,\kappa_1)dx}.\label{eq:RBM_lip_estim_2}
\end{align}
First
\begin{align*}
    \left(\int_\mathbb{R} g(x,\sigma,\kappa_1)dx\right)^{-1}\leq \left(\int_\mathbb{R} \exp\left(-\frac{1}{\sigma_0}\left(U(x)+\frac{L_W}{2}x^2-L_Wx\kappa_1+C\right)\right)dx\right)^{-1}.
\end{align*}
Then, for all $x\in\mathbb{R}$ and all $\alpha\geq0$, we have
\begin{align*}
    U(x)+\frac{L_W}{2}x^2-L_Wx\kappa_1+C=\frac{x^4}{4}-\frac{x^2}{2}+\frac{1}{4}+\frac{L_W}{2}|x-\kappa_1|^2\geq \alpha x^2-\beta_{\alpha},
\end{align*}
with 
\begin{align*}
    \beta_{\alpha}=&\frac{(2\alpha+1)^2}{4}-\frac{1}{4}=\alpha^2+\alpha.
\end{align*}
Thus, for all integers $k\geq0$, we have
\begin{align*}
    \int_\mathbb{R}x^{2k}\exp\left(-\frac{1}{\sigma}\left(U(x)+\frac{L_W}{2}x^2-L_Wx\kappa_1+C\right)\right)dx\leq& e^{\frac{\alpha^2+\alpha}{\sigma}} \int_\mathbb{R}x^{2k}\exp\left(-\frac{\alpha x^2}{\sigma}\right)dx\\
    =&e^{\frac{\alpha^2+\alpha}{\sigma}}\sqrt{2\pi}\frac{(2k)!}{2^kk!}\left(\frac{\sigma}{2\alpha}\right)^{k+\frac{1}{2}}.
\end{align*}
Choosing $\alpha=\frac{\sqrt{\sigma}}{2}$, we obtain
\begin{align*}
    \int_\mathbb{R}x^{2k}\exp\left(-\frac{1}{\sigma}\left(U(x)+\frac{L_W}{2}x^2-L_Wx\kappa_1+C\right)\right)dx\leq&e^{\frac{1}{4}+\frac{1}{2\sqrt{\sigma}}}\sqrt{2\pi}\frac{(2k)!}{2^kk!}\sigma^{\frac{k}{2}+\frac{1}{4}}.
\end{align*}
Hence there exists $C$, independent of $\sigma$, such that for $k\leq\frac{7}{2}$, $\frac{1}{\sigma^2}\int_\mathbb{R}x^{2k}g(\sigma,x)dx\leq C$ (which allows us to deal with \eqref{eq:RBM_lip_estim_1})
and for $k\leq2$, $\frac{1}{\sigma^{5/4}}\int_\mathbb{R}x^{2k}g(\sigma,x)dx\leq C$ and for $k\leq1$, $\frac{1}{\sigma^{3/4}}\int_\mathbb{R}x^{2k}g(\sigma,x)dx\leq C$ (both allow us to deal with \eqref{eq:RBM_lip_estim_2}). Thus, for $\sigma\geq \sigma_0$, $\left|\partial_\sigma f_2(\sigma,\kappa_1)\right|$ is bounded uniformly in $\kappa_1$, which yields the result.
\end{proof}

\paragraph{We now show that $\sigma\mapsto \kappa_2(\mu_{\sigma,+})$ is Lipschitz continuous.}

Let $\sigma_c>\sigma_0>0$. We have already proved, in Lemma~\ref{lem:var_lip}, that $\sigma\mapsto f_2(\sigma, \kappa_1)$ is Lipschitz continuous uniformly in $\kappa_1\in[-C_{\kappa_1},C_{\kappa_1}]$. However, the difficulty lies in the fact that, for $\kappa_2(\mu_{\sigma,+})$ given by $\kappa_2(\mu_{\sigma,+})=f_2(\sigma, \kappa_1(\mu_{\sigma,+}))$, the mean $\sigma\mapsto \kappa_1(\mu_{\sigma,+})$ is not Lipschitz continuous around $\sigma_c$. We will work our way around this fact (and, doing so, also prove it) in the rest of the subsection, but in the meantime this can be numerically observed in Figure~\ref{Fig:analyse_tugaut}.

Let us compute the various derivatives of $f_1$ and $f_2$ given in \eqref{eq:def_f1} and \eqref{eq:def_f2}.
\begin{align*}
    \partial_\sigma g(x,\sigma,\kappa)=&\frac{\left(U(x)+\frac{L_W}{2}|x-\kappa|^2\right)}{\sigma^2}g(x,\sigma,\kappa),\\
    \partial_\kappa g(x,\sigma,\kappa)=&\frac{L_W}{\sigma}(x-\kappa)g(x,\sigma,\kappa),\\
    \partial_\sigma f_1(\sigma,\kappa)=&\frac{1}{\left(\int_\mathbb{R} g\right)^2}\left(\int_\mathbb{R} g\int_\mathbb{R}x\partial_\sigma g-\int_\mathbb{R} \partial_\sigma g\int_\mathbb{R}xg\right)\\
    =&\frac{1}{\sigma^2}\left(\frac{\int_\mathbb{R}x\left(U(x)+\frac{L_W}{2}|x-\kappa|^2\right) g}{\int_\mathbb{R} g}-\frac{\int_\mathbb{R}\left(U(x)+\frac{L_W}{2}|x-\kappa|^2\right) g}{\int_\mathbb{R} g}\frac{\int_\mathbb{R}xg}{\int_\mathbb{R} g}\right),\\
    \partial_\kappa f_1(\sigma,\kappa)=&\frac{1}{\left(\int_\mathbb{R} g\right)^2}\left(\int_\mathbb{R} g\int_\mathbb{R}x\partial_\kappa g-\int_\mathbb{R} \partial_\kappa g\int_\mathbb{R}xg\right)\\
    =&\frac{L_W}{\sigma}\left(\frac{\int_\mathbb{R}x(x-\kappa) g}{\int_\mathbb{R} g}-\frac{\int_\mathbb{R}(x-\kappa) g}{\int_\mathbb{R} g}\frac{\int_\mathbb{R}xg}{\int_\mathbb{R} g}\right)\\
    =&\frac{L_W}{\sigma}\left(\frac{\int_\mathbb{R}x^2 g}{\int_\mathbb{R} g}-\left(\frac{\int_\mathbb{R}x g}{\int_\mathbb{R} g}\right)^2\right),\\
    \partial_\sigma f_2(\sigma,\kappa)=&\frac{1}{\left(\int_\mathbb{R} g\right)^2}\left(\int_\mathbb{R} g\int_\mathbb{R}(x-\kappa)^2\partial_\sigma g-\int_\mathbb{R} \partial_\sigma g\int_\mathbb{R}(x-\kappa)^2g\right)\\
    =&\frac{1}{\sigma^2}\left(\frac{\int_\mathbb{R}(x-\kappa)^2\left(U(x)+\frac{L_W}{2}|x-\kappa|^2\right) g}{\int_\mathbb{R} g}-\frac{\int_\mathbb{R}\left(U(x)+\frac{L_W}{2}|x-\kappa|^2\right) g}{\int_\mathbb{R} g}\frac{\int_\mathbb{R}(x-\kappa)^2g}{\int_\mathbb{R} g}\right),\\
    \partial_\kappa f_2(\sigma,\kappa)=&\frac{1}{\left(\int_\mathbb{R} g\right)^2}\left(\left(2\int_\mathbb{R}(\kappa-x)g+\int_\mathbb{R}(x-\kappa)^2\partial_\kappa g\right)\int_\mathbb{R} g-\int_\mathbb{R} \partial_\kappa g\int_\mathbb{R}(x-\kappa)^2g\right)\\
    =&\frac{2\int_\mathbb{R}(\kappa-x)g}{\int_\mathbb{R} g}+\frac{L_W}{\sigma}\left(\frac{\int_\mathbb{R}(x-\kappa)^3g}{\int_\mathbb{R} g}-\frac{\int_\mathbb{R}(x-\kappa)^2g}{\int_\mathbb{R} g}\frac{\int_\mathbb{R}(x-\kappa)g}{\int_\mathbb{R} g}\right).
\end{align*}
In particular, notice that 
\begin{equation}\label{eq:RBM_der_part}
    \partial_\kappa  f_1(\sigma,\kappa_1(\mu_{\sigma,*}))=\frac{L_W}{\sigma}\kappa_2(\mu_{\sigma,*})\ \ \ \text{ and }\ \ \ \partial_\kappa  f_2(\sigma,\kappa_1(\mu_{\sigma,*}))=\frac{L_W}{\sigma}\frac{\int_\mathbb{R}(x- \kappa_1(\mu_{\sigma,*}))^3g}{\int_\mathbb{R} g}.
\end{equation}
We have
\begin{align*}
    \frac{d}{d\sigma}\kappa_1(\mu_{\sigma,*})=&\partial_\sigma f_1(\sigma,\kappa_1(\mu_{\sigma,*}))+\partial_\kappa  f_1(\sigma,\kappa_1(\mu_{\sigma,*}))\frac{d}{d\sigma}\kappa_1(\mu_{\sigma,*}),\\
    \frac{d}{d\sigma}\kappa_2(\mu_{\sigma,*})=&\partial_\sigma f_2(\sigma,\kappa_1(\mu_{\sigma,*}))+\partial_\kappa  f_2(\sigma,\kappa_1(\mu_{\sigma,*}))\frac{d}{d\sigma}\kappa_1(\mu_{\sigma,*}).
\end{align*}
Thus
\begin{align}\label{eq:dm1}
    \frac{d}{d\sigma}\kappa_1(\mu_{\sigma,*})=\frac{\partial_\sigma f_1(\sigma,\kappa_1(\mu_{\sigma,*}))}{1-\partial_\kappa  f_1(\sigma,\kappa_1(\mu_{\sigma,*}))}=\frac{\partial_\sigma f_1(\sigma,\kappa_1(\mu_{\sigma,*}))}{1-\frac{L_W}{\sigma}\kappa_2(\mu_{\sigma,*})}.
\end{align}
By the results on the critical variance in Lemma~\ref{lem:res_NL}, $\kappa_1(\mu_{\sigma,*})$ is continuously differentiable on $]\sigma_0,\sigma_c[$. Likewise
\begin{align}
    \frac{d}{d\sigma}\kappa_2(\mu_{\sigma,*})&\left(1-\frac{L_W}{\sigma}\kappa_2(\mu_{\sigma,*})\right)\nonumber\\
    &=\left(1-\frac{L_W}{\sigma}\kappa_2(\mu_{\sigma,*})\right)\partial_\sigma f_2(\sigma,\kappa_1(\mu_{\sigma,*}))+\partial_\kappa  f_2(\sigma,\kappa_1(\mu_{\sigma,*}))\partial_\sigma f_1(\sigma,\kappa_1(\mu_{\sigma,*})).\label{eq:RBM_lip_estim_3}
\end{align}
The fact that $1-\frac{L_W}{\sigma}\kappa_2(\mu_{\sigma,*})$ goes to 0 as $\sigma\rightarrow\sigma_c$ is what prevents us from giving an upper bound on $\frac{d}{d\sigma}\kappa_1(\mu_{\sigma,*})$. \plb{By the result on the critical variance in Lemma~\ref{lem:res_NL} and by continuity, there may be a problem in the limit $\sigma\rightarrow\sigma_c^-$, but we wan already say that $\sigma\mapsto \kappa_2(\mu_{\sigma,+})$ is Lipschitz continuous on any interval of the form $[\sigma_1,\sigma_2]$ with $0<\sigma_1<\sigma_2<\sigma_c$. The following lemma gives a more precise speed of convergence of the mean to $0$ around the critical parameter.}

%
%Lemma
%

\begin{lemma}\label{lem:m1_racine}
There exists $C>0$ such that
\begin{align*}
    \frac{\kappa_1(\mu_{\sigma,+})}{\sqrt{\sigma_c-\sigma}}\xrightarrow[\sigma\rightarrow\sigma_c^-]{}C.
\end{align*}
\end{lemma}

%Proof

\begin{proof}
\pierre{We restrict the study to $\sigma\in[\sigma_0,\sigma_c[$ for some arbitrary $\sigma_0>0$. } We compute
\begin{align*}
    \partial_\kappa  f_1(\sigma,\kappa)=&\frac{L_W}{\sigma}\left(\frac{\int_\mathbb{R}x^2 g}{\int_\mathbb{R} g}-\left(\frac{\int_\mathbb{R}x g}{\int_\mathbb{R} g}\right)^2\right),\\
    \partial_{\kappa}^2f_1(\sigma,\kappa)=&\left(\frac{L_W}{\sigma}\right)^2\left(\frac{\int_\mathbb{R}x^3 g}{\int_\mathbb{R} g}-3\frac{\int_\mathbb{R}x^2 g}{\int_\mathbb{R} g}\frac{\int_\mathbb{R}x g}{\int_\mathbb{R} g}+2\left(\frac{\int_\mathbb{R}x g}{\int_\mathbb{R} g}\right)^3\right),\\
    \partial_{\kappa}^3f_1(\sigma,\kappa)=&\left(\frac{L_W}{\sigma}\right)^3\left(\frac{\int_\mathbb{R}x^4 g}{\int_\mathbb{R} g}-4\frac{\int_\mathbb{R}x^3 g}{\int_\mathbb{R} g}\frac{\int_\mathbb{R}x g}{\int_\mathbb{R} g}\plb{+12\frac{\int_\mathbb{R}x^2 g}{\int_\mathbb{R} g}\left(\frac{\int_\mathbb{R}x g}{\int_\mathbb{R} g}\right)^2-}6\left(\frac{\int_\mathbb{R}x g}{\int_\mathbb{R} g}\right)^4-3\left(\frac{\int_\mathbb{R}x^2 g}{\int_\mathbb{R} g}\right)^2\right).
\end{align*}
In particular
\begin{align*}
    f_1(\sigma,0)=&0,\\
    \partial_\kappa  f_1(\sigma,0)=&\frac{L_W}{\sigma}\frac{\int_\mathbb{R}x^2 g}{\int_\mathbb{R} g}>0,\\
    \partial_{\kappa}^2f_1(\sigma,0)=&0,
\end{align*}
and, by Lemma~\ref{lem:der_3}
\begin{equation}\label{eq:der_3}
    \partial_{\kappa}^3f_1(\sigma,0)=\left(\frac{L_W}{\sigma}\right)^3\left(\frac{\int_\mathbb{R}x^4 g}{\int_\mathbb{R} g}-3\left(\frac{\int_\mathbb{R}x^2 g}{\int_\mathbb{R} g}\right)^2\right)\xrightarrow[]{\sigma\rightarrow\sigma_c^-}\partial_{\kappa}^3f_1(\sigma_c,0)<0.
\end{equation}
Let us start by proving that there exists $C>0$ such that, in the limit $\sigma\rightarrow\sigma_c$ (or equivalently $\kappa_1(\mu_{\sigma,+})\rightarrow0$), we have
\begin{align}\label{eq:intermediaire}
    \frac{\sigma_c-\sigma}{\kappa_1(\mu_{\sigma,+})^2}<C+o(1).
\end{align}
We compute
\begin{align*}
    \partial_\kappa f_1(\sigma,\kappa)=\partial_\kappa f_1(\sigma,0)+\kappa\partial^2_{\kappa}f_1(\sigma,0)+\frac{\kappa^2}{2}\partial^3_{\kappa}f_1(\sigma,0)+o(\kappa^2).
\end{align*}
By Lemma~\ref{lem:d_sigma_d_m_xi}, there exists $C>0$ such that
\begin{align*}
    \partial_\kappa f_1(\sigma,0)\geq\partial_\kappa f_1(\sigma_c,0)+C(\sigma_c-\sigma)=1+C(\sigma_c-\sigma).
\end{align*}
Since $\partial_\kappa  f_1(\sigma_c,0)=\partial_\kappa  f_1(\sigma_c,\kappa_1(\mu_{\sigma_c}))=\frac{L_W}{\sigma_c}\kappa_2(\mu_{\sigma_c})=1$ by \eqref{eq:val_var_crit}
\begin{align*}
    \partial_\kappa f_1(\sigma,\kappa)\geq1+C(\sigma_c-\sigma)+\frac{\kappa^2}{2}\left(\partial^3_{\kappa}f_1(\sigma_c,0)+o_{\sigma\rightarrow\sigma_c}(1)\right)+o(\kappa^2),
\end{align*}
where the $o(\kappa^2)$ is uniform in $\sigma\in[\sigma_0,\sigma_c]$. Considering $\kappa=\kappa_1(\mu_{\sigma,+})$, which goes to 0 as $\sigma\rightarrow\sigma_c$, in the equation above yields
\begin{align*}
    \frac{L_W}{\sigma}\kappa_2(\mu_{\sigma,+})\geq 1+C(\sigma_c-\sigma)+\frac{\kappa_1(\mu_{\sigma,+})^2}{2}\left(\partial^3_{\kappa}f_1(\sigma_c,0)+o_{\sigma\rightarrow\sigma_c}(1)\right)+o(\kappa_1(\mu_{\sigma,+})^2).
\end{align*}
By the results of Lemma~\ref{lem:res_NL} concerning the critical variance, $\frac{L_W}{\sigma}\kappa_2(\mu_{\sigma,+})\leq1$, which gives
\begin{align*}
    0\geq \frac{C(\sigma_c-\sigma)}{\kappa_1(\mu_{\sigma,+})^2}+\frac{\partial^3_{\kappa}f_1(\sigma_c,0)}{2}+o_{\sigma\rightarrow\sigma_c}(1).
\end{align*}
This gives \eqref{eq:intermediaire} and this in turns allows us to state that $\sigma_c-\sigma=O(\kappa_1(\mu_{\sigma,+})^2)$.\\

We then have
\begin{align*}
    f_1(\sigma,\kappa)=&f_1(\sigma,0)+\kappa\partial_\kappa  f_1(\sigma,0)+\frac{\kappa^2}{2}\partial_{\kappa}^2f_1(\sigma,0)+\frac{\kappa^3}{6}\partial_{\kappa}^3f_1(\sigma,0)+o(\kappa^3)\\
    =&\kappa\partial_\kappa  f_1(\sigma,0)+\frac{\kappa^3}{6}\partial_{\kappa}^3f_1(\sigma,0)+o(\kappa^3).
\end{align*}
Furthermore, defining $F_1$ as in Lemma~\ref{lem:d_sigma_d_m_xi}, we obtain
\begin{align*}
    \partial_\kappa  f_1(\sigma,0)=F_1(\sigma)=&F_1(\sigma_c)-(\sigma_c-\sigma)F'_1(\sigma_c)+\plb{o(\sigma_c-\sigma)}\\
    =&1-(\sigma_c-\sigma)F'_1(\sigma_c)+o(\sigma_c-\sigma),
\end{align*}
which then yields
\begin{align*}
    f_1(\sigma,\kappa)=&\kappa\left(1-(\sigma_c-\sigma)F'_1(\sigma_c)+o(\sigma_c-\sigma)\right)+\frac{\kappa^3}{6}\left(\partial_{\kappa}^3f_1(\sigma_c,0)+o(1)\right)+o(\kappa^3).
\end{align*}
Since $\kappa_1(\mu_{\sigma,+})=f_1(\sigma,\kappa_1(\mu_{\sigma,+}))$, we thus get in the limit $\sigma\rightarrow\sigma_c^-$
\begin{align}
    0=&-\kappa_1(\mu_{\sigma,+})(\sigma_c-\sigma)F'_1(\sigma_c)+\frac{\kappa_1(\mu_{\sigma,*})^3}{6}\partial_{\kappa}^3f_1(\sigma_c,0)\\
    &+o(\kappa_1(\mu_{\sigma,+})^3)+\kappa_1(\mu_{\sigma,+})o(\sigma_c-\sigma)\nonumber\\
    0=&-(\sigma_c-\sigma)F'_1(\sigma_c)+\frac{\kappa_1(\mu_{\sigma,+})^2}{6}\partial_{\kappa}^3f_1(\sigma_c,0)+o(\kappa_1(\mu_{\sigma,+})^2)\label{eq:comp_ordre_sigma_k_1}.
\end{align}
 Thus, thanks to \eqref{eq:der_3} and Lemma~\ref{lem:d_sigma_d_m_xi}, there exists $C>0$ such that 
\begin{align*}
    \frac{\sigma_c-\sigma}{\kappa_1(\mu_{\sigma,+})^2}\xrightarrow[]{\sigma\rightarrow\sigma_c^-}C,
\end{align*}
which yields the final result.
\end{proof}

%
%Lemma
%

\begin{lemma}
Let $\sigma_0\in]0,\sigma_c[$. Then $\sigma \mapsto \kappa_2(\mu_{\sigma,+})$ is Lispchitz continuous on $[\sigma_0,\sigma_c]$.
\end{lemma}

%Proof

\begin{proof}
\plb{
Let us write for all $(\sigma,\kappa)\in[\sigma_0,\sigma_c]\times[-C_{\kappa_1},C_{\kappa_1}]$
\begin{align*}
    \partial_\kappa  f_1(\sigma,\kappa)=&\partial_\kappa  f_1(\sigma,0)+\kappa \partial^2_\kappa  f_1(\sigma,0)+\frac{\kappa^2}{2} \partial^3_\kappa  f_1(\sigma,0)+o(\kappa^3)\\
    =&\left(F_1(\sigma_c)-(\sigma_c-\sigma)F'_1(\sigma_c)+o(\sigma_c-\sigma)\right)+\frac{\kappa^2}{2} \left(\partial^3_\kappa  f_1(\sigma_c,0)+O(\sigma_c-\sigma)\right)+o(\kappa^3),
\end{align*}
where we used the notation $F_1$ from Lemma~\ref{lem:d_sigma_d_m_xi}, the fact that $\partial^2_\kappa  f_1(\sigma,0)=0$, and where all notation $o(\cdot)$ and $O(\cdot)$ are uniform in $(\sigma,\kappa)\in[\sigma_0,\sigma_c]\times[-C_{\kappa_1},C_{\kappa_1}]$ by continuity. Since $F_1(\sigma_c)=1$, we obtain
\begin{align*}
1-\partial_\kappa  f_1(\sigma,\kappa)=(\sigma_c-\sigma)F'_1(\sigma_c)-\frac{\kappa^2}{2}\partial^3_\kappa  f_1(\sigma_c,0)+o(\sigma_c-\sigma)+o(\kappa^3)+\pierre{\kappa^2} O(\sigma_c-\sigma).
\end{align*}
Applying this for $\kappa=\kappa_1(\mu_{\sigma,+})$, and using \eqref{eq:comp_ordre_sigma_k_1} \pierre{ and Lemma~\ref{lem:m1_racine}}, we obtain
\begin{align*}
1-\partial_\kappa  f_1(\sigma,\kappa_1(\mu_{\sigma,+}))=&\frac{\kappa_1(\mu_{\sigma,+})^2}{6}\partial^3_\kappa  f_1(\sigma_c,0)-\frac{\kappa_1(\mu_{\sigma,+})^2}{2}\partial^3_\kappa  f_1(\sigma_c,0)+o(\kappa_1(\mu_{\sigma,+})^2)\\
=&-\frac{\kappa_1(\mu_{\sigma,+})^2}{3}\partial^3_\kappa  f_1(\sigma_c,0)+o(\kappa_1(\mu_{\sigma,+})^2).
\end{align*}
Hence, using \eqref{eq:dm1}, for all $\sigma\in[\sigma_0,\sigma_c[$
\begin{align*}
    \left|\frac{d}{d\sigma}\kappa_1(\mu_{\sigma,+})\right|=\frac{|\partial_\sigma f_1(\sigma,\kappa_1(\mu_{\sigma,+}))|}{|1-\partial_\kappa  f_1(\sigma,\kappa_1(\mu_{\sigma,+}))|}\lesssim \frac{C\kappa_1(\mu_{\sigma,+})}{\kappa_1(\mu_{\sigma,+})^2+o(\kappa_1(\mu_{\sigma,+})^2)} \lesssim \frac{1}{\kappa_1(\mu_{\sigma,+})}+o(\kappa_1(\mu_{\sigma,+})),
\end{align*}
where we used that $\partial_\sigma f_1(\sigma,0)=0$ \pierre{and that $\partial_\kappa\partial_\sigma f_1$ is bounded over $[\sigma_0,\sigma_c]\times[-C_{\kappa_1},C_{\kappa_1}]$ to bound the numerator, and Lemma~\ref{lem:der_3} for the denominator}.
}
Besides, since $\partial_\kappa f_2(\sigma,0)=0$ (by \eqref{eq:RBM_der_part} and symmetry),
\begin{align*}
    \left|\partial_\kappa  f_2(\sigma, \kappa_1(\mu_{\sigma,+})\right|=&\left|\partial_\kappa  f_2(\sigma, \kappa_1(\mu_{\sigma,+}))-\partial_\kappa  f_2(\sigma, 0)\right|\lesssim |\kappa_1(\mu_{\sigma,+})|\plb{+o(\kappa_1(\mu_{\sigma,+}))}
\end{align*}
and thus we bound for all $\sigma \in [\sigma_0,\sigma_c[$
\begin{align*}
    \left|\frac{d}{d\sigma}\kappa_2(\mu_{\sigma,+})\right|\leq&\left|\partial_\sigma f_2(\sigma, \kappa_1(\mu_{\sigma,+}))\right|+\left|\frac{d}{d\sigma}\kappa_1(\mu_{\sigma,+})\right|\left|\partial_\kappa  f_2(\sigma, \kappa_1(\mu_{\sigma,+}))\right|\plb{+o(\kappa_1(\mu_{\sigma,+}))}\\
    \lesssim&1+ \frac{\left|\partial_\kappa  f_2(\sigma, \kappa_1(\mu_{\sigma,+})\right|}{\kappa_1(\mu_{\sigma,+})}\plb{+o(1)}\\
    \lesssim& 1\plb{+o(1)}.
\end{align*}
\plb{By continuity (recall that $1-\partial_\kappa  f_1(\sigma,\kappa_1(\mu_{\sigma,+}))=1-\frac{L_W}{\sigma}\kappa_2(\mu_{\sigma,+})>0$ for $\sigma<\sigma_c$) , }this proves that $\sigma \mapsto \kappa_2(\mu_{\sigma,+})$  is Lipschitz on $[\sigma_0,\sigma_c]$.
\end{proof}

%
%Section
%

\section*{Acknowledgements}
This work has been (partially) supported by the Project EFI ANR-17-CE40-0030 of the French National Research Agency. P. Monmarché's research is supported by the Project SWIDIMS ANR-20-CE40-0022 of the French National Research Agency.

\bibliographystyle{alpha}
\bibliography{bibliographie}

\end{document}